\documentclass[final,1p]{elsarticle}

\usepackage{appendix}
\usepackage[colorlinks,linkcolor=blue, anchorcolor=blue, citecolor=blue]{hyperref}
\usepackage{amssymb}
\usepackage{amsmath}
\usepackage{amsthm}
\usepackage{color,tabularx,booktabs}
\usepackage{graphicx,subfigure} 
\biboptions{numbers,sort&compress}

\newcommand{\f}{\frac}
\newcommand{\p}{\partial}

\newcommand{\me}{\mathcal{ E }}

\newcommand{\mg}{\mathcal{G}}
\newcommand{\mo}{\mathcal{O}}
\newcommand{\md}{\mathcal{D}}

\newtheorem{theorem}{Theorem}[section]
\newtheorem{lemma}[theorem]{Lemma}

\theoremstyle{definition}

\theoremstyle{remark}
\newtheorem{remark}[theorem]{Remark}

\numberwithin{equation}{section}
\allowdisplaybreaks

\journal{Journal of Scientific Computing}
\begin{document}

\begin{frontmatter}

\title{Energy Dissipation Law and Maximum Bound Principle-Preserving Linear  BDF2 Schemes with Variable Steps for the Allen--Cahn Equation }
\author[ouc]{Bingyin Zhang} \ead{zhangbingyin@stu.ouc.edu.cn}
\author[ouc,lab]{Hongfei Fu\corref{cor1}} \ead{fhf@ouc.edu.cn}
\author[ouc,lab]{Rihui Lan} \ead{lanrihui@ouc.edu.cn}
\author[ouc,lab]{Shusen Xie} \ead{shusenxie@ouc.edu.cn}

\cortext[cor1]{Corresponding author.}

\affiliation[ouc]{organization={School of Mathematical Sciences},
	addressline={Ocean University of China}, 
	city={Qingdao},
	postcode={266100}, 
	state={Shandong},
	country={China}}
\affiliation[lab]{organization={Laboratory of Marine Mathematics},
	addressline={Ocean University of China}, 
	city={Qingdao},
	postcode={266100}, 
	state={Shandong},
	country={China}}

\begin{abstract}
In this paper, we propose and analyze a linear, structure-preserving  method for solving the Allen--Cahn equation based on the second-order backward differentiation formula (BDF2) with variable time steps and the scalar auxiliary variable (SAV) approach.  To this end, we first design a novel and essential auxiliary functional that serves twofold functions: (i) ensuring that a first-order approximation to the auxiliary variable, which is essentially important for deriving the unconditional energy dissipation law, does not affect the second-order temporal accuracy of the phase function $\phi$; and (ii) enabling the development of effective stabilization terms that facilitate the construction of maximum bound principle (MBP)-preserving linear methods. Combining this novel functional with a standard central difference stencil, we then propose a linear, second-order variable-step BDF2 type stabilized exponential SAV scheme, referred to as BDF2-sESAV-I. The scheme is proven to preserve both the discrete modified energy dissipation law under the temporal stepsize ratio $ 0 < r_{k} := \tau_{k}/\tau_{k-1} < 4.864 - \delta $ with a positive constant $\delta$ and the MBP under $ 0 < r_{k} < 1 + \sqrt{2} $. Moreover,  the approximation of the original energy by the modified energy is analyzed. By leveraging the kernel recombination technique,  optimal $ H^{1}$- and $ L^{\infty}$-norm error estimates for the variable-step BDF2-sESAV-I scheme are rigorously established. Numerical examples are carried out to verify the theoretical results and demonstrate the effectiveness and efficiency of the proposed scheme.
\end{abstract}


\end{frontmatter}

\section{Introduction}
To model the motion of the anti-phase boundaries in crystalline solids, Allen \& Cahn proposed the following classical Allen--Cahn equation \cite{Acta_Allen_1979}
\begin{equation}\label{Model:tAC}
	\begin{aligned}
		\p_t \phi   = \varepsilon^2 \Delta \phi + f( \phi ), \quad t >0, \  \mathbf{x} \in \Omega,
	\end{aligned}
\end{equation}
with initial condition $ \phi( \mathbf{x}, 0 ) = \phi_{\text {init }} ( \mathbf{x} ) $ for any $ \mathbf{x} \in \bar{\Omega} $, where $\Omega \subset \mathbb{R}^d (d = 1,2,3)$ is a bounded
Lipschitz domain, $ \phi( \mathbf{x}, t ) $ is the difference between the concentrations of two
components of the alloy, $  0 < \varepsilon \ll 1 $ is the interaction length that describes the thickness of the transition boundary between materials, and $ f(\phi)  $ is a continuously differentiable nonlinear reaction. Equipped with the suitable boundary condition, such as periodic or homogeneous Neumann boundary condition, model \eqref{Model:tAC} can be viewed as an $ L^{2} $ gradient flow associating with the free energy
\begin{equation*}\label{def:energy}
	\begin{aligned}
		E[\phi] := \int_{\Omega} \Big( \frac{\varepsilon^2}{2} \vert \nabla \phi(\mathbf{x}) \vert^2 
		+ F( \phi(\mathbf{x}) ) \Big) d \mathbf{x},
	\end{aligned}
\end{equation*}
where $F$ represents the general nonlinear potential function satisfying $ - F' = f $, see \eqref{poten:dw} and \eqref{poten:fh} for two popular choices of potentials. Recently, some variants of \eqref{Model:tAC} have been developed to describe various processes of phase transition, such as the mass-conserving Allen--Cahn equation \cite{IMA_1992_Rubinstein} and the nonlocal Allen--Cahn equation \cite{FIC_2006_Bates,SINUM_Du_2019}.

The solution to \eqref{Model:tAC} satisfies the so-called \textit{energy dissipation law} in the sense that the energy $ E[\phi] $ monotonically decreases along with the time, i.e., $ \frac{ d }{ dt } E[ \phi(t) ] \leq 0 $. Furthermore, if the following condition is imposed
\begin{equation}\label{Condi:MBP}
	\text{there exists a constant} \  \beta > 0 \  \text{such that} \  f(\beta) \leq 0 \leq f (-\beta),
\end{equation}
model \eqref{Model:tAC} also satisfies the \textit{maximum bound principle} (MBP),  i.e.,
\begin{equation}\label{MBP}
	\max_{\mathbf{x} \in \bar{\Omega}} |\phi_{\text {init }}(\mathbf{x})| \leq \beta \quad \Longrightarrow 
	\quad \max_{\mathbf{x} \in \bar{\Omega}}|\phi(\mathbf{x},t)| \leq \beta, \quad \forall t>0,
\end{equation}
see \cite{SIREV_Du_2021} for details.
Therefore, to achieve stable numerical simulations and avoid nonphysical solutions, it is highly desirable for numerical methods to preserve the discrete energy dissipation law and MBP. That is also the main objective of this paper.

Another key feature of the Allen--Cahn equation is that the dynamics process usually takes a long time to reach the steady state. Moreover, due to the nonlinear interaction, the associated energy decays quickly at the early stage of the dynamics, while it decays rather slowly when it reaches a steady state \cite{MOC_2023_Ju,SINUM_2020_Liao}. This recommends the  
\textit{non-uniform temporal grids}, which necessitates that the time-stepping method considered is theoretically reliable, i.e., stable, convergent and preferably structure-preserving. 
However, it is highly nontrivial and much challenging for numerical analysis of variable-step multi-step time integration methods, for example, the BDF2 method \cite{BIT_1998_Becker,ESAIM_2024_Fu}. Due to its strong stability \cite{SINUM_1974_Gear,MOC_2021_Liao,SISC_1997_Shampine}, the variable-step BDF2 method has been frequently employed in recent years to establish energy-stable numerical schemes for phase-field model, including, but not limited to, the Allen--Cahn equation, see 
\cite{SINUM_2020_Liao,SINUM_2019_Chen,JCAM_2024_Zhao,IMA_2022_Liao,JSC_2023_Chen,MCS_2023_Li,CNSNS_2024_Mei,ANM_2025_Zhang}. Specifically, for model \eqref{Model:tAC} with double-well potential, the authors in 
\cite{SINUM_2020_Liao} proposed a fully-implicit BDF2 scheme, and established the discrete energy dissipation law under $ r_{k} < 3.561 $ and MBP under $ r_{k} < 1 + \sqrt{2} $ by developing a kernel recombination technique. Meanwhile, second-order convergence in the maximum-norm was derived with the help of discrete complementary convolution kernels. This is the first paper about a nonlinear variable-step BDF2 scheme that can preserve both the energy dissipation law and MBP simultaneously. Employing a similar approach, Hu et al. \cite{JCAM_2024_Zhao} investigated the space-fractional Allen--Cahn equation by using the variable-step BDF2 method, and established the discrete energy dissipation law under $ r_{k} < 4.864 $ and MBP under $ r_{k} < 1 + \sqrt{2} $. However, it must be pointed out that the idea of \cite{SINUM_2020_Liao} cannot be directly extended to construct a linear scheme, and of course, the fully-implicit method, in which a nonlinear iteration must be implemented at each time level, may result in expensive computational costs. To improve the computational efficiency, and maintain the inherent physical properties, it is essential to develop a linear, structure-preserving BDF2 scheme with variable time steps. Therefore, this work can be regarded as an essential improvement of the fully-implicit BDF2 scheme \cite{SINUM_2020_Liao}. 

In the past decades, a large amount of research has been devoted to modeling the phase field models with preservation of the discrete energy dissipation law, including convex splitting methods \cite{SINUM_2013_Baskaran,MRS_1998_Eyre}, stabilized implicit-explicit methods \cite{DCDS_2010_Shen,SINUM_2006_Xu}, exponential time differencing (ETD) methods \cite{SIREV_Du_2021,MOC_2018_Ju}, invariant energy quadratization (IEQ) methods \cite{JCP_2016_Yang}, scalar auxiliary variable (SAV) methods \cite{SINUM_2018_Shen,JCP_2018_Shen,SIREV_2019_Shen}, etc. Notably, the SAV methods are more convenience in designing linear and unconditionally energy dissipative schemes, see \cite{SISC_2019_Akrivis,BIT_2022_Kemmochi,SISC_2020_Liu} and the references therein. In \cite{JCP_2018_Shen,SIREV_2019_Shen}, Shen et al. proposed an efficient BDF2-SAV scheme that is unconditionally energy dissipative with respect to a modified energy. Very recently, the authors in \cite{JSC_Qiao_2023} presented a variable-step BDF2-SAV scheme and demonstrated its unconditional energy dissipation law and temporal convergence. However, the MBP is neither considered nor proved in the aforementioned BDF2-based SAV schemes.  Actually, the construction and analysis of MBP-preserving SAV schemes, particularly BDF2-based numerical methods, is not trivial. This complexity arises from factors such as the uncertainty in the signs of the nonlinear term coefficients, as well as negativity and non-monotonicity of the discrete BDF2 coefficients.

Recently, the preservation of MBP has attracted increasingly attention in the development of numerical methods for the Allen–Cahn equation \eqref{Model:tAC}. In \cite{SISC_2020_LiBuyang}, the authors proposed a class of high-order MBP-preserving methods based on the exponential integrator in time, and lumped mass finite element method in space, where values exceeding the maximum-bound at nodal points are eliminated by a cut-off approach. Combining the SAV method with the cut-off postprocessing technique, Yang et al. \cite{JSC_2022_Yang} presented and analyzed a class of energy dissipative and MBP-preserving schemes. Doan et al. presented first- and second-order low regularity integrators (LRI) schemes by iteratively using Duhamel’s formula in \cite{BIT_2023_Schratz}, where the discrete energy stability and MBP of the second-order LRI scheme were established by imposing that $ f(\phi) $ is twice continuously differentiable and nonconstant on $ [-\beta, \beta] $.
By introducing an artificial stabilization term, Du et al. developed first-order ETD and second-order ETD Runge-Kutta (ETDRK2) schemes for the nonlocal Allen--Cahn equation in \cite{SINUM_Du_2019}, which preserves the discrete MBP unconditionally. Subsequently, they established an abstract framework on MBP-preservation of the ETD methods for a class of semilinear parabolic equations \cite{SIREV_Du_2021}. Borrowing the idea of stabilized  ETD approach, several MBP-preserving SAV methods were successfully proposed, see \cite{SINUM_Ju_2022,JSC_Ju_2022}. Very recently, this stabilized technique was also employed to construct a linear variable-step BDF2 scheme in \cite{MOC_2023_Ju}, where the preservation of the MBP was established, but unfortunately, the discrete energy dissipation law was not considered. Instead, the authors established an energy stability result for the scheme under $ r_{k} < 1 + \sqrt{2} $, in the sense that the discrete energy is uniformly bounded by the initial energy plus a constant, see \cite[Theorem 4.2 \& Remark 4.2]{MOC_2023_Ju}. Actually, to the best of our knowledge, there is no existing linear BDF2 scheme, whether on uniform or non-uniform temporal grids, that can simultaneously preserve the energy dissipation law and MBP for the Allen--Cahn equation \eqref{Model:tAC}.  

Our main goal is to construct and analyze linear, energy dissipative, and MBP-preserving BDF2 schemes with variable steps for the Allen--Cahn equation \eqref{Model:tAC} using the exponential SAV (ESAV) approach, in which a first-order method (in time) is used to approximate the auxiliary variable. To this end, we introduce a novel and crucial auxiliary functional in subsection \ref{subsec:AuxF}, which ensures that the first-order approximation of the auxiliary variable does not affect the second-order accuracy of $ \phi $. More importantly, together with the variable-step BDF2 method, such an auxiliary functional allows us to design stabilized MBP-preserving ESAV schemes with different stabilization terms: one unbalanced and two balanced. Among these, the unbalanced scheme is recommended due to its excellence in preserving discrete MBP for large temporal stepsize. In particular, for the recommended scheme, the following main theoretical results are established:
	\begin{itemize}
		\item The scheme is \textit{unconditionally energy dissipative} under the temporal stepsize ratio $ 0< r_{k} < 4.8645 - \delta $ and \textit{MBP}-preserving for $ 0< r_{k} < 1 + \sqrt{2} $. To the best of our knowledge, this is the first time that a linear (variable-step) BDF2 scheme has been developed with provable preservation of both the energy dissipation law and MBP in discrete settings;
		
		\item \textit{Optimal error estimates} for the phase function in $H^{1}$- and $ L^{\infty} $-norm are discussed by employing the kernel recombination technique developed in 
		\cite{MOC_2023_Ju,SINUM_2020_Liao,SISC_2016_Xu};
		
		\item The difference between \textit{the original energy} and \textit{the modified energy} is analyzed, demonstrating that the modified energy serves as an approximation of the original energy and is \textit{uniformly stable} with respect to the initial energy. This appears to be the first time such a result has been established for an SAV-type method.
		
	\end{itemize}

The remainder of the paper is organized as follows. 
In Section \ref{sec:Scheme}, a novel and essential auxiliary functional is designed, by which three stabilized ESAV schemes with different stabilization terms based on the variable-step BDF2 formula are presented.  In Section \ref{sec:ED_MBP}, the discrete energy dissipation law and MBP of the proposed schemes, along with an analysis of the approximation to the original energy by the modified one, are established. Then, we present the optimal-order error estimates in the $H^{1}$- and $ L^{\infty} $-norm in Section \ref{sec:Error}. Numerical experiments are carried out to illustrate the performance of the proposed schemes in Section \ref{sec:NumTest}. Some concluding remarks are finally drawn in the last section. 

\section{A novel linear second-order variable-step BDF2-sESAV scheme}\label{sec:Scheme}
In this section, we first present some notations related to the temporal and spatial discretizations, 
and then construct linear stablized BDF2 schemes for \eqref{Model:tAC} based on a novel auxiliary functional. For the simplicity of presentation, we focus on the two-dimensional square domain $ \Omega=(0,L)^2$ for the Allen--Cahn equation equipped with periodic boundary conditions throughout this paper. It is easy to extend the corresponding results to the three-dimensional case and/or homogeneous Neumann boundary condition.

\subsection{Preliminaries}
To construct numerical schemes with variable time steps, we consider a nonuniform temporal partition $ 0 = t_{0} < t_{1} < \cdots < t_{N} = T$ with stepsize $ \tau_{k} := t_{k} - t_{k-1} $ for $ 1 \leq k \leq N $. Denote by $ \tau := \max_{1 \leq k \leq N} \tau_{k}$ the maximum temporal stepsize and by $ r_{k} := \tau_{k}/\tau_{k-1}$ $ (k \geq 2) $ with $ r_{1} \equiv 0 $ the adjacent temporal stepsize ratio. Let $r_{\max }:=\max_{1 \leq k \leq N}r_{k}$. For any real sequence $ \{ w^{k} \}^{N}_{k=0} $, set $ \nabla_{\tau} w^{k} = w^{k} - w^{k-1}$ and $ \md_{1} w^{k} = \nabla_{\tau} w^{k}/\tau_{k} $. The variable-step BDF2 formula is defined as
$$
\md_{2} w^{n} := \frac{ 1 + 2 r_{n} }{\tau_{n} (1 + r_{n}) } \nabla_{\tau} w^{n} - \frac{ r_{n}^{2} }{\tau_{n} (1 + r_{n}) } \nabla_{\tau} w^{n-1}\quad \text{for } \   n \ge 2.
$$
In particular, when $n=1$,  we use the BDF1 (i.e., backward Euler) formula $ \md_{1} w^{1}$ for the first time level discretization.
Then, we rewrite the above variable-step BDF formula as a unified discrete convolution summation
\begin{equation}\label{Formula:BDF2}
	\begin{aligned}
		\md_{2} w^{n} = \sum^{n}_{k=1} b^{(n)}_{n-k} \nabla_{\tau} w^{k}\quad \text{for } \  1 \le n \le N,
	\end{aligned}
\end{equation}
where $\{  b^{(n)}_{n-k} \}_{k=1}^{n}$ are called the discrete convolution kernels, given by $ b^{(1)}_{0} := 1/\tau_{1}$ for $n =1$, and for $n \ge 2$,
\begin{equation}\label{Def:DCK}
	\begin{aligned}
		b^{(n)}_{0} := \frac{ 1 + 2 r_{n} }{\tau_{n} (1 + r_{n}) } > 0, ~~ b^{(n)}_{1} := - \frac{ r_{n}^{2} }{\tau_{n} (1 + r_{n}) } < 0, ~~  b^{(n)}_{j} := 0 \   \   \text{for} \,\  2 \leq j \leq n-1.
	\end{aligned}
\end{equation}

The following lemma shows the positiveness of the kernels $ \{b^{(n)}_{n-k}\} $, which plays a key role in the proof of the discrete energy dissipation law and error analysis.
\begin{lemma}\label{lem:BDF2_P1} 
	Assume the stepsize ratio condition $ 0 < r_k < r_{\max} := 4.864 - \delta$ for $2 \leq k \leq n$ holds, where $ 0< \delta < 4.864$ is any given small constant. Then, the convolution kernels $\{  b^{(n)}_{n-k} \}_{k=1}^{n}$ are positive definite, that is, for any real sequence $\left\{w_k\right\}_{k=0}^n$, it holds
	\begin{equation}\label{BDF:Pd}
		\begin{array}{l}
			w_n \sum\limits_{k=1}^n b_{n-k}^{(n)} w_k \geq 
			\left\{
			\begin{aligned}
				& G[ w^n ] - G[ w^{n-1} ] + \frac{\delta w_n^2}{32 \tau_n}, & n \geq 2 ,\\
				& G[ w^{1} ],  &  n = 1, 
			\end{aligned}
			\right.
		\end{array}
	\end{equation}
	where $ G[ w^n ] := \frac{r_{n+1} \sqrt{r_{\max }}}{2\left(1+r_{n+1}\right)} \frac{w_n^2}{\tau_n} $.
\end{lemma}
\begin{proof}
The conclusion \eqref{BDF:Pd} for $ n \geq 2 $ can be found in \cite[Lemma 3.1]{JSC_Zhang_2022}. It remains to consider the case $ n = 1 $, and it is sufficient to verify the inequality $ \frac{r_{2} \sqrt{r_{\max }}}{2\left(1+r_{2}\right)} \leq 1 $, i.e., $ ( \sqrt{ r_{\max} } - 2 ) r_{2} \leq 2 $. Obviously, if $ r_{\max} \leq 4 $, this condition holds true for any $ 0 < r_{2} < r_{\max} $. Otherwise, it requires $ r_{2} \leq \frac{ 2 }{ \sqrt{ r_{\max} } - 2 } $ holds for $ 4 < r_{\max} < 4.864 $, which sufficiently needs $ r_{2} \leq \frac{ 2 }{ \sqrt{ 4.864 } - 2 } \approx 9.7348 $. In summary, \eqref{BDF:Pd} holds for all  $ 0 < r_k < r_{\max}$, ~$k \ge 2$.
\end{proof}

Given a positive integer $ M $, let $ h = L/M $ be the spatial grid length and set $ \Omega_{h} := \{ \mathbf{x}_{h} = ( ih, jh ) \mid 0 \leq i,j \leq M \} $. Let $ \mathbb{V}_h $ be the set of all $M$-periodic real-valued grid functions on $ \Omega_{h} $, i.e.,
$$
\mathbb{V}_h:=\left\{v \mid v= \{v_{i,j}\}_{i,j=1}^{M}~ \text{and}~  v ~ \text{is periodic}\right\} .
$$ 
Define the following discrete inner product, and discrete $ L^2 $ and $ L^{\infty} $ norms
$$
\langle v, w\rangle=h^2 \sum_{i, j=1}^M v_{i j} w_{i j}, \quad\|v\|=\sqrt{\langle v, v\rangle}, \quad\|v\|_{\infty}=\max _{1 \leq i, j \leq M} |v_{i j}|
$$
for any $v, w \in \mathbb{V}_h$, and
$$
\langle \mathbf{v}, \mathbf{w}\rangle = \langle v^1, w^1\rangle + \langle v^2, w^2\rangle,  \quad \|\mathbf{v}\|=\sqrt{\langle\mathbf{v}, \mathbf{v}\rangle}
$$
for any $\mathbf{v} = (v^1, v^2), \mathbf{w}= (w^1, w^2) \in \mathbb{V}_h \times \mathbb{V}_h$. 
Moreover, for any $ v \in \mathbb{V}_h $, we define the discrete Laplace operator $\Delta_h$ and gradient operator $\nabla_h$ as follows
$$
\Delta_h v_{i j}=\frac{1}{h^2}\left(v_{i+1, j}+v_{i-1, j}+v_{i, j+1}+v_{i, j-1}-4 v_{i j}\right), \quad 1 \leq i, j \leq M,
$$
$$
\nabla_h v_{i j}=\left(\frac{v_{i+1, j}-v_{i j}}{h}, \frac{v_{i, j+1}-v_{i j}}{h}\right), \quad 1 \leq i, j \leq M .
$$
Notice that $ \mathbb{V}_h $ is a finite-dimensional linear space, thus any grid function in $ \mathbb{V}_h $ and any linear operator $ P: \mathbb{V}_h \rightarrow \mathbb{V}_h $ can be treated as a vector in $ \mathbb{R}^{M^2} $ and a matrix
in $ \mathbb{R}^{M^2 \times M^2 } $, respectively. One can readily verify that the discrete matrix $\Delta_h = (d_{ij})$ is symmetric and negative semidefinite, and that its diagonal entries satisfy $ d_{ii} = - \max_{i} \sum_{ j \neq i } \vert d_{ij} \vert  $.

\begin{lemma}[\cite{SINUM_2020_Liao}]\label{lem:MBP_left} 
	 Let $ B $ be a real $ m \times m $ matrix. If the elements of $ B = ( b_{ij} ) $ fulfill $ b_{ii} = - \max_{i} \sum_{ j \neq i } \vert b_{ij} \vert  $, then for any $ a >0 $ and $ v \in \mathbb{R}^{m} $, we have
	 $$
	 \| ( aI - B ) v \|_{\infty} \geq a \| v \|_{\infty},
	 $$
	 where $ I $ represents the identity matrix.
\end{lemma}

\begin{lemma}[\cite{SIREV_Du_2021}]\label{lem:MBP_right} 
	Under assumption \eqref{Condi:MBP}, if $\kappa \geq\left\|f^{\prime}\right\|_{C[-\beta, \beta]}$ holds for some positive constant $\kappa$, then we have $|f( \xi ) + \kappa \xi | \leq \kappa \beta$ for any $\xi \in[-\beta, \beta]$.
\end{lemma}

\subsection{An essential auxiliary functional}\label{subsec:AuxF}
In \cite{JSC_Qiao_2023}, the authors introduced an auxiliary functional $ W(z) $ of $ z $ satisfying
\begin{equation*}\label{AuxF_Hou}
	W (1) = 1, \quad \lim_{z\rightarrow1} \frac{ W (z) - 1 }{ 1 - z } = -1,
\end{equation*}
that is, $ W'(1) = -1 $. It is easy to check that $ z W(z) $ is a second-order approximation to $1$, if $ z $ approximates $ 1 $ with first-order accuracy. This functional ensures that their proposed SAV scheme is second-order, although the auxiliary variable has only first-order accuracy, see \cite{JSC_Qiao_2023} for details. Similar ideas can also be found in \cite{CMAME_Huang_2022,SISC_Huang_Shen_2020}.

However, it must be pointed out that such a functional cannot guarantee the non-negativity and boundedness of $ z W(z) $, which may be detrimental to the construction of MBP-preserving numerical methods. To address this issue, we introduce a novel class of auxiliary functional $ V( \cdot ) $, which satisfy the following two assumptions:
\begin{itemize}
	\item[(\bf A1).]  $ V(\cdot) \in C^{1} ( \mathbb{R} ) \cap W^{2,\infty} ( \mathbb{R} )$ such that $ V(1) = 1 $, $ V'(1) = 0 $, and 
	there exists a positive constant $ K $ such that $ | V' ( \cdot ) | \leq K  $;
	\item[(\bf A2).] It is non-negative and bounded by $1$, i.e., $ 0 \leq V( \cdot ) \leq 1  $.
\end{itemize}

Such functional exists, and a simple example can be constructed using the well-known cubic Hermite interpolation, e.g.,
\begin{equation}\label{f_cut}
	\begin{array}{l}
		V(z) :=
		\left\{
		\begin{aligned}
			& 0, & z \in (-\infty, 0] ,\\
			& H_1 ( z ),  &  z \in (0,1/2),          \\
			& 2z - z^2,  &    z \in [1/2, 3/2],\\				
			& H_2 ( z ),  &  z \in (3/2,2),       \\
			& 0, &  z \in [2,\infty),
		\end{aligned}
		\right.
	\end{array}
\end{equation}
where $ H_1 ( z ) $ and $ H_2 ( z ) $ be the cubic Hermite interpolation on $ (0,1/2) $ and $ (3/2,2) $ respectively, i.e.,
$$
H_1 ( z ) =  - 8 z^3 + 7 z^2 > 0, ~ z \in (0,1/2), ~~
H_2 ( z ) = 8 z^3 - 41 z^2 + 68 z - 36 > 0, ~ z \in (3/2,2).
$$
It is easy to check that the above function \eqref{f_cut} satisfies the assumptions (\textbf{A1})--(\textbf{A2}) with $ K = 49/24 $; see Fig. \ref{fig_1}.
\begin{figure}[!htbp]
	\vspace{-6pt}
	\centering
	\begin{minipage}[t]{0.5\linewidth}
		\centering
		\includegraphics[width=2.5in]{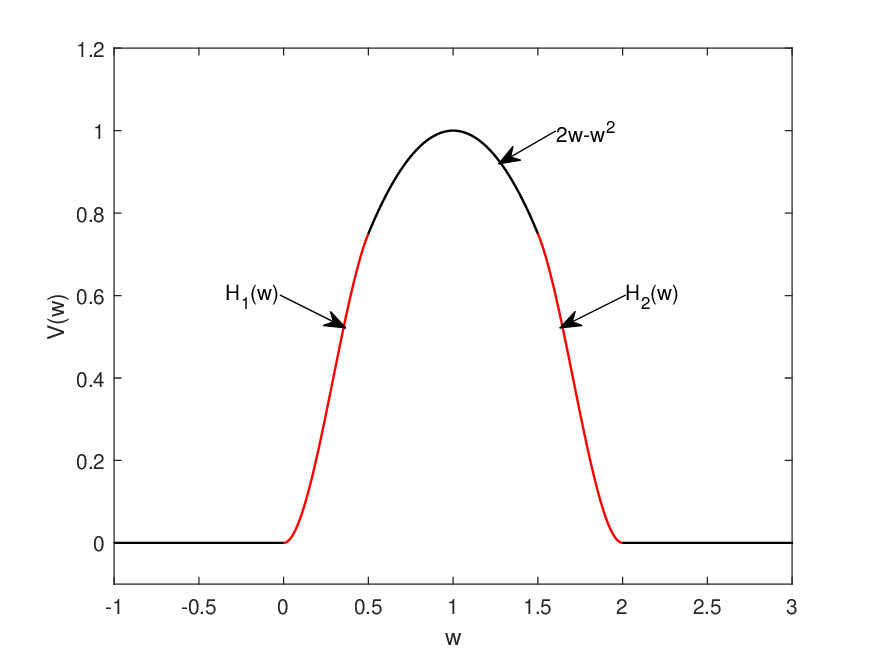}
		\caption{ Plot of $ V(z) $ on the interval $ [-1,3] $ }
		\label{fig_1}
	\end{minipage}
	\vspace{-6pt}
\end{figure}
\begin{remark}
	In the stabilized ESAV scheme to be designed, the assumption (\textbf{A1}) on the functional $ V(\cdot) $ ensures that the first-order approximation of the auxiliary variable does not affect the second-order convergence of the phase function $ \phi $. The non-negativity and boundedness assumption (\textbf{A2}) allows us to develop several different stabilization terms, and thus MBP-preserving numerical methods.
\end{remark}

\subsection{Variable-step BDF2-sESAV schemes}
Let us introduce an auxiliary variable $ R(t) = E_1[\phi] := \int_{\Omega}  F( \phi(\mathbf{x},t) )  d \mathbf{x}$. Then, the original energy $E[\phi]$ is equivalent to
$$
\me[\phi, R] = \frac{\varepsilon^2}{2}\|\nabla \phi\|_{L^2}^2 + R(t).
$$
Denote $g(\phi, R):=\frac{\exp \left\{R\right\}}{ \exp \left\{E_{1}[\phi]\right\} }$, which is always equal to 1 at the continuous level. Then, model \eqref{Model:tAC} can be rewritten into the following equivalent system:
\begin{equation}\label{sav_1}
		\phi_t  =\varepsilon^2 \Delta \phi + V( g( \phi, R ) ) f(\phi),~~
		R_t  = -V( g( \phi, R ) ) \left(f(\phi), \phi_t\right) .
\end{equation}

Let $ \{ \phi^n, R^n \} $ be the fully discrete approximations of the exact solutions $ \{ \phi( t_{n} ), R( t_n ) \} $ to the original problem \eqref{sav_1}. 
For $ n \geq 1 $, we first compute a predicted solution $ \hat{\phi}^{n} $ using a linear first-order BDF1 scheme with the temporal stepsize $ \tau_{n} $, given by
\begin{equation}\label{sch:2_3}
		\frac{\hat{\phi}^{n} - \phi^{n-1}}{ \tau_{n} }  = \varepsilon^2 \Delta_h \hat{\phi}^{n} + f(\phi^{n-1}) - \kappa ( \hat{\phi}^{n} - \phi^{n-1}) .
\end{equation}
To develop a stabilized fully-discrete scheme, we introduce a modified and \textit{unbalanced} stabilization term
\begin{equation}\label{def:ST}
	\kappa \big( \phi^{n} - V ( g_h(\hat{\phi}^{n}, R^{n-1}) ) \hat{\phi}^{n} \big).
\end{equation}
Here $ \kappa \geq 0 $ is a stabilization constant, and $ g_h(v, w):= \frac{\exp \left\{w\right\}}{ \exp \left\{E_{1h}[v]\right\} } $ with $E_{1 h}$ being the discrete version of $E_1$, i.e., $E_{1 h}(v):=\langle F(v), 1\rangle$. The imbalance, reflected in $ V ( g_h(\hat{\phi}^{n}, R^{n-1}) ) \neq 1 $ for $ n \geq 1 $ generally, is beneficial for preservation of the MBP.

Then, a novel variable-step BDF2 type stabilized ESAV scheme (denoted as BDF2-sESAV-I) is given by
\begin{equation}\label{sch:2_1}
	\left\{
	\begin{aligned}
		&\md_{2} \phi^{n} = \varepsilon^2 \Delta_h \phi^{n} + V( g_h(\hat{\phi}^{n}, R^{n-1}) ) f(\hat{\phi}^{n}) 
		- \kappa \big( \phi^{n} - V ( g_h(\hat{\phi}^{n}, R^{n-1}) ) \hat{\phi}^{n} \big), \\
		&\md_{1} R^{n}  = - \big\langle V ( g_h(\hat{\phi}^{n}, R^{n-1}) ) f(\hat{\phi}^{n}) 
		- \kappa \big( \phi^{n} - V ( g_h(\hat{\phi}^{n}, R^{n-1}) ) \hat{\phi}^{n} \big), \md_{1} \phi^{n} \big\rangle,
	\end{aligned}
	\right.
\end{equation}
with initial values $ \{ \phi^0, R^0 \} = \{ \phi_{\text{init}}, E_{1h}[\phi_{\text{init}}] \} $. Note that the first equation in \eqref{sch:2_1} is equivalent to the following form:
\begin{equation*}\label{sch:2_4}
		(( b^{(n)}_{0} + \kappa ) I - \varepsilon^2 \Delta_h) \phi^{n} = b^{(n)}_{0} \phi^{n-1} - \sum^{n-1}_{k=1} b^{(n)}_{n-k} \nabla_{\tau} \phi^{k} +  V( g_h(\hat{\phi}^{n}, R^{n-1}) ) \big( f(\hat{\phi}^{n}) + \kappa \hat{\phi}^{n} \big) .
\end{equation*}
It is easy to see that $ ( b^{(n)}_{0} + \kappa ) I - \varepsilon^2 \Delta_h $ is self-adjoint and positive definite, which implies the unique solvability of the system \eqref{sch:2_1}.


\begin{remark}\label{def:other_S} 
	For the Allen--Cahn type equation \eqref{Model:tAC}, there are two other \textit{balanced} stabilization terms often used to devise stabilized MBP-preserving schemes, i.e.,
	\begin{itemize}
		\item $ \kappa ( \phi^{n} - \hat{\phi}^{n} ) $: considered in fully-implicit \cite{SINUM_2020_Liao} and implicit-explicit \cite{MOC_2023_Ju} BDF2-type schemes;
		
		\item $ \kappa V ( g_h(\hat{\phi}^{n}, R^{n-1}) ) ( \phi^{n} - \hat{\phi}^{n} ) $: proposed in stabilized exponential integrator \cite{SINUM_Ju_2022} and Crank-Nicolson \cite{JSC_Ju_2022} SAV schemes with a different functional $ V ( z ) := z $.
	\end{itemize}
	Together with these two stabilization terms, instead of \eqref{def:ST}, one can obtain the other two variable-step BDF2-sESAV schemes: 
	\begin{equation*}\label{sch:2_2_1}
			\left\{
		\begin{aligned}
			&\md_{2} \phi^{n}  = \varepsilon^2 \Delta_h \phi^{n} + V( g_h(\hat{\phi}^{n}, R^{n-1}) ) f(\hat{\phi}^{n}) - \kappa ( \phi^{n} - \hat{\phi}^{n} ),\\
			&\md_{1} R^{n}  = - \big\langle V ( g_h(\hat{\phi}^{n}, R^{n-1}) ) f(\hat{\phi}^{n}) - \kappa(\phi^{n} - \hat{\phi}^{n}), \md_{1} \phi^{n} \big\rangle,
		\end{aligned}
			\right.
	\end{equation*}
	and  
	\begin{equation*}\label{sch:2_3_1}
			\left\{
		\begin{aligned}
			& \md_{2} \phi^{n}  = \varepsilon^2 \Delta_h \phi^{n} + V( g_h(\hat{\phi}^{n}, R^{n-1}) ) f(\hat{\phi}^{n}) - \kappa V ( g_h(\hat{\phi}^{n}, R^{n-1}) ) ( \phi^{n} -  \hat{\phi}^{n} ),\\
			& \md_{1} R^{n}  = - V ( g_h(\hat{\phi}^{n}, R^{n-1}) ) \big\langle f(\hat{\phi}^{n}) - \kappa ( \phi^{n} - \hat{\phi}^{n} ), \md_{1} \phi^{n} \big\rangle,
		\end{aligned}
			\right.
	\end{equation*}
	which are abbreviated as BDF2-sESAV-II  and BDF2-sESAV-III below.
	It is easy to check that both schemes are uniquely solvable. We will observe in the following Remark \ref{rem:MBP} and subsection \ref{sec:ex2} that the BDF2-sESAV-I scheme \eqref{sch:2_1}  with the novel proposed stabilization term has a significant advantage in preserving MBP, compared to BDF2-sESAV-II  and BDF2-sESAV-III. Therefore, we restrict ourselves to the discussion of the variable-step BDF2-sESAV-I scheme in the rest of this paper. 
\end{remark}

The discrete MBP for the BDF1 scheme \eqref{sch:2_3} is listed in the following lemma, which plays a significant role in the proof of the MBP for the BDF2-sESAV-I scheme.
\begin{lemma}[\cite{JCM_Tang_2016,CMS_Shen_2016}]\label{lem:MBP_BDF1} 
	Assume that $\|\phi^{n-1}\|_{\infty} \leq \beta$ and the stabilization parameter $
	\kappa \geq \left\|f^{\prime}\right\|_{C[-\beta, \beta]},
	$
	then it unconditionally holds for the BDF1 scheme \eqref{sch:2_3} that $\|\hat{\phi}^{n}\|_{\infty} \leq \beta$ for $n=1, \cdots, N$.
\end{lemma}

\subsection{Approximation error of $ V ( g_h(\hat{\phi}^{n}, R^{n-1}) ) -1$}

It is obvious that the numerical approximations to $ R(t) $ in \eqref{sch:2_1} is only first-order accurate in time. One may guess that  the BDF2-sESAV-I scheme for the evaluation of $\phi$ is also only of first-order accurate.  In fact, we can find from \eqref{sch:2_1} that the SAV $ R $ affect the phase function $\phi$ by terms $ V( g_h(\hat{\phi}^{n}, R^{n-1}) ) f(\hat{\phi}^{n}) $ and 
$$
\kappa \big( \phi^{n} - V ( g_h(\hat{\phi}^{n}, R^{n-1}) ) \hat{\phi}^{n} \big) = \kappa \big( \phi^{n} - \hat{\phi}^{n} \big) + \kappa \big( 1 - V ( g_h(\hat{\phi}^{n}, R^{n-1}) ) \big) \hat{\phi}^{n},
$$
which imply that the error of $\phi$ is affected by $ R $ only through the term $ 1 - V ( g_h(\hat{\phi}^{n}, R^{n-1}) ) $. Below, we  show that a first-order approximation to the auxiliary variable does not affect the second-order temporal accuracy of $\phi$. 

\begin{lemma}\label{lem:estig} 
	Assume $ R^{n} \leq \mathcal{M} $ with a positive constant $ \mathcal{M} > 0 $ and $ \| \phi^{n} \|_{\infty} \leq \beta $ hold for all $ n \geq 0 $. Denote $ \varsigma_n := \| \phi(t_{n}) - \hat{\phi}^{n} \| + \vert R( t_{n-1} ) - R^{n-1} \vert + \tau_{n} + h^2 $. Then, under assumption (\textbf{A1}) we have
	\begin{equation*}
			\big\vert V( g_h( \hat{\phi}^n, R^{n-1} ) ) - 1 \big\vert \leq C \min \left\{ \varsigma_n,  \varsigma_n ^2  \right\}.
	\end{equation*}
\end{lemma}
\begin{proof} Due to assumption (\textbf{A1}), a direct application of Taylor expansion gives us
	\begin{equation*}
		\begin{aligned}
			V( g_h( \hat{\phi}^n, R^{n-1} ) ) -1 & = \int_{1}^{ g_h( \hat{\phi}^n, R^{n-1} ) } V^{\prime} ( \theta ) d\theta, \\  
			V( g_h( \hat{\phi}^n, R^{n-1} ) )-1 & = \int_{1}^{ g_h( \hat{\phi}^n, R^{n-1} ) } \big(  g_h( \hat{\phi}^n, R^{n-1} )-\theta \big) V^{\prime\prime}  ( \theta ) d\theta,
		\end{aligned}
	\end{equation*}
	where $ V(1) = 1 $ and $ V^{\prime}(1) = 0 $ have been applied.  Thus, to derive the conclusion, we have to estimate $\big\vert g_h( \hat{\phi}^n, R^{n-1} ) - 1 \big\vert$ below.
It follows from the triangle inequality that
\begin{equation}\label{estig_1}
	\begin{aligned}
		&  \big\vert g_h( \hat{\phi}^n, R^{n-1} ) - 1 \big\vert \\
		& \quad \leq \big\vert g_h( \hat{\phi}^n, R^{n-1} ) - g_h( \hat{\phi}^n, R( t_{n-1} ) ) \big\vert 
		+ \big\vert g_h( \hat{\phi}^n, R( t_{n-1} ) ) - g_h( \phi( t_{n} ), R( t_{n-1} ) ) \big\vert \\
		& \quad \quad + \big\vert  g_h( \phi( t_{n} ), R( t_{n-1} ) )  - g_h( \phi( t_{n} ), R( t_{n} ) ) \big\vert + \big\vert g_h( \phi( t_{n} ), R( t_{n} ) ) - 1 \big\vert \\ 
		& \quad =: S_{1}+S_{2}+S_{3} + S_{4},
	\end{aligned}
\end{equation}
in which, by the boundedness of $ R^{n} $ and $ \phi^{n} $, Lemma \ref{lem:MBP_BDF1}, and Cauchy-Schwarz inequality, we see
\begin{equation*}
		S_{1}  = \frac{ 1 }{ \exp \{E_{1 h}[ \hat{\phi}^{n} ] \} } \big\vert \exp \{ R( t_{n-1} ) \} - \exp \{ R^{n-1} \}  \big\vert  \leq C\vert R( t_{n-1} ) - R^{n-1} \vert,
\end{equation*}
\begin{equation*}
	\begin{aligned}
		S_{2} &   = \frac{ \exp \{ R( t_{n-1} ) \} }{ \exp \{E_{1 h}[ \phi( t_n ) ] \} \exp \{E_{1 h}[ \hat{\phi}^{n} ] \} } \big\vert \exp \{E_{1 h}[ \phi( t_n ) ] \} -  \exp \{E_{1 h}[ \hat{\phi}^{n} ] \} \big\vert \\
		& \leq \frac{ \exp \{ R( t_{n-1} ) \} \vert \Omega \vert^{1/2} }{ 
			\min\{ \exp \{E_{1 h}[ \phi( t_n ) ] \}, \exp \{E_{1 h}[ \hat{\phi}^{n} ] \} \} } \| F( \phi( t_{n} ) ) - F( \hat{\phi}^{n} ) \| \leq C \| \phi(t_{n}) - \hat{\phi}^{n}  \|.
	\end{aligned}
\end{equation*}
From the second equation of \eqref{sav_1}, we have $ \vert R_t \vert = \vert (f(\phi), \phi_t) \vert \leq C (\| f( \phi ) \|^2 + \| \phi_t \|^2) $, which together with the mean value theorem gives
$
		S_{3} \leq C \big\vert R( t_{n} ) - R( t_{n-1} ) \big\vert   =\mo(\tau_{n}).
$
Finally, by  definition  we have
\begin{equation*}
		S_{4}  = \frac{ \big\vert \exp \{ R( t_{n} ) \} - \exp \{ E_{1 h}[ \phi( t_{n} ) \} ] \big\vert }{ \exp \{E_{1 h}[ \phi( t_n ) ] \} }   \leq C \big\vert E_{1} [ \phi(t_n) ] - E_{1h} [ \phi(t_n) ] \big\vert =\mo( h^2).
\end{equation*}
Therefore, inserting the above estimates into \eqref{estig_1} we obtain
	\begin{equation*}
		\big\vert g_h( \hat{\phi}^n, R^{n-1} ) - 1 \big\vert  \leq C \varsigma_n,
\end{equation*}
which further implies that
$$
\vert V( g_h( \hat{\phi}^n, R^{n-1} ) ) - 1 \vert  \leq C \varsigma_{n}, \quad  \vert V( g_h( \hat{\phi}^n, R^{n-1} ) ) - 1 \vert  \leq C \varsigma_{n}^2.
$$
Thus, the conclusion is proved.
\end{proof}

\begin{remark}
Note that the proof of Lemma \ref{lem:estig}  requires that $ R^{n} $ and $ \| \phi^{n} \|_{\infty} $ to be bounded by $ \mathcal{M} $ and $ \beta $, respectively. Actually, they can be deduced by the discrete MBP (cf. Theorem \ref{thm:MBP_2}) and energy dissipation law (cf. Theorem \ref{thm:energy_2}) in the next section. Fig. \ref{figAV} shows the approximation of $ g_h(\hat{\phi}, R) $ and $ V ( g_h(\hat{\phi}, R) ) $ yielded by the BDF2-sESAV-I scheme to $1$ in subsection \ref{sec:ex2}, from which we can observe that the use of auxiliary functional $ V( \cdot ) $ defined in subsection \ref{subsec:AuxF} can indeed improve the accuracy.
\begin{figure}[!htbp]
	\vspace{-12pt}
	\centering
	\subfigure[double-well potential with fixed $\tau = 0.04$]
	{
		\includegraphics[width=0.45\textwidth]{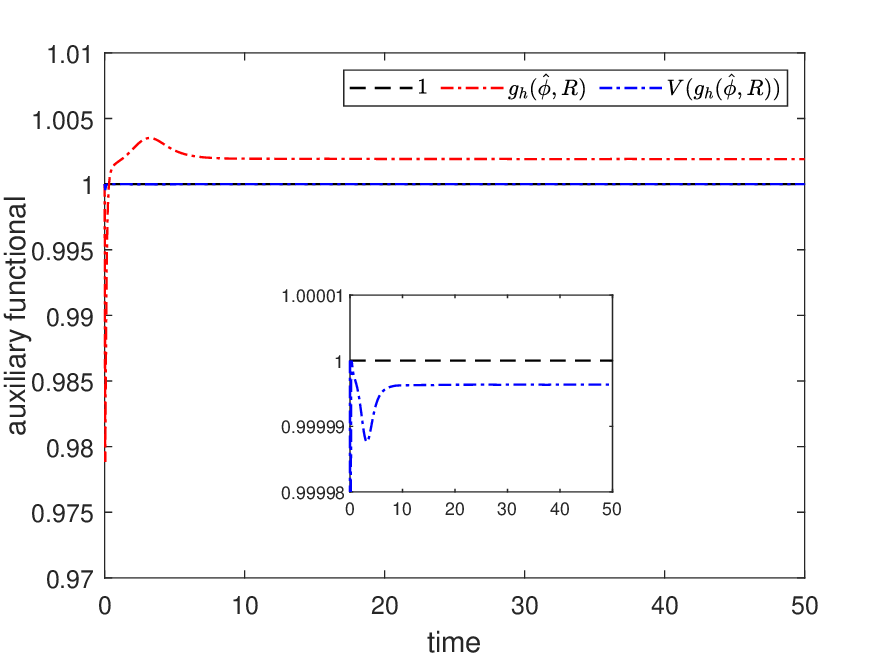}		\label{figAV_dw}
	}%
	\subfigure[Flory--Huggins potential with fixed $\tau = 0.01$]
	{
		\includegraphics[width=0.45\textwidth]{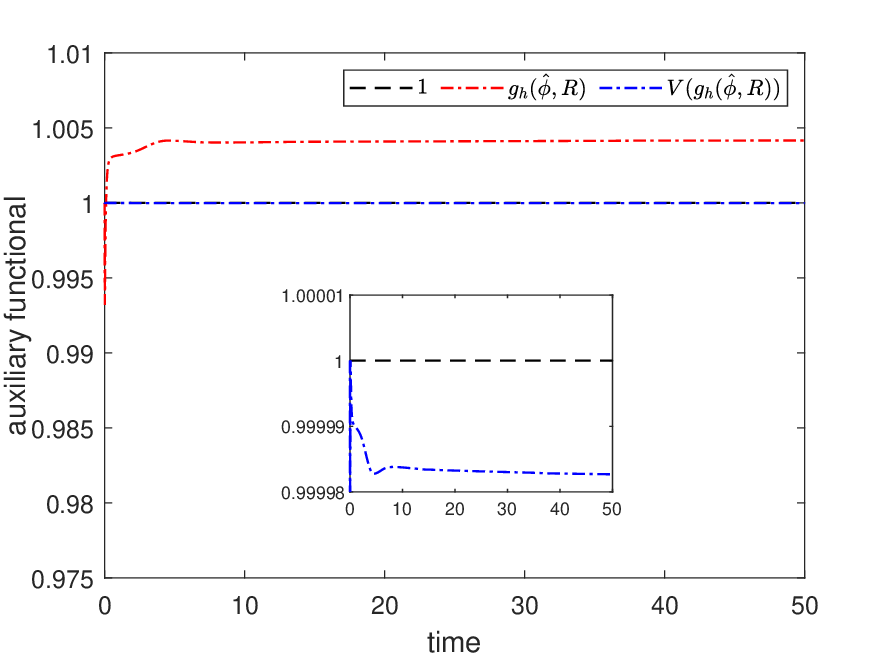}	\label{figAV_fh}
	}%
	\setlength{\abovecaptionskip}{0.0cm} 
	\setlength{\belowcaptionskip}{0.0cm}
	\caption{Approximations of $ g_h(\hat{\phi}, R) $ and $ V ( g_h(\hat{\phi}, R) ) $ to $1$}	
	\label{figAV}
\end{figure}
\end{remark}

\section{Discrete energy dissipation law and MBP}\label{sec:ED_MBP}
In this section, we investigate the discrete energy dissipation law and MBP of the proposed BDF2-sESAV-I scheme by using a kernel recombination technique.

\subsection{Discrete maximum bound principle}\label{subsec:mbp}
One of the main difficulties in analyzing the discrete MBP and error estimate of the proposed scheme is related to the fact that the coefficients of the variable-step BDF2 formula is non-monotonic and may change sign. We are devoted to address the issue by using a kernel recombination technique. Define a new variable $ \psi $ as
\begin{equation*}\label{def:new_ker}
	\psi^{0} := \phi^0 \quad \text{and} \quad \psi^{k} := \phi^{k} - \eta \phi^{k-1} \quad   \text{  for  } \  k \geq 1,
\end{equation*}
where $\eta$ is a real parameter to be determined. It is easy to derive that
	\begin{equation}\label{new_ker:f1}
		\phi^{k} = \sum_{\ell=0}^{k} \eta^{k-\ell} \psi^\ell, \quad \nabla_\tau \phi^{k} = \sum_{\ell=1}^k \eta^{k-\ell} \nabla_\tau \psi^{\ell} + \eta^{k} \phi^0 \quad   \text{  for  } \  k \geq 1.
\end{equation}


Inserting \eqref{new_ker:f1} into \eqref{Formula:BDF2} and exchanging the summation order, we get the following  updated variable-step BDF2 formula
\begin{equation}\label{Formula:new_BDF2}
	\md_{2} \phi^n = \sum^{n}_{k=1} d^{(n)}_{n-k} \nabla_{\tau} \psi^{k} + d^{(n)}_{n} \psi^0 = d_0^{(n)} \psi^{n} - \sum^{n-1}_{k=0} ( d^{(n)}_{n-k-1} - d^{(n)}_{n-k} ) \psi^{k} \quad \text{for} \  n \geq 1,
\end{equation}
where the new discrete convolution kernels $ \{d^{(n)}_{n-k}\} $ are defined as follows: $ d_0^{(n)} := b_0^{(n)} $ and
\begin{equation}\label{Formula:new_ker_1}
	 d_{1}^{(1)} := \eta d_{0}^{(1)}, \quad  d_k^{(n)} := \eta^{k-1} (b_0^{(n)} \eta + b_1^{(n)} ) \quad \text{for} \  1 \leq k \leq n \  \text{with} \   n \geq 2,
\end{equation}
which gives
\begin{equation}\label{Formula:new_ker_2}
	 d_{k}^{(n)} = \eta d_{k-1}^{(n)} \quad  \text{for} \  1 \leq k \leq n.
\end{equation}
To make $ d^{(n)}_{n-k} $ non-negative and decreasing, i.e., $d_0^{(n)} \geq d_1^{(n)} \geq \cdots \geq d_n^{(n)} \geq 0$, we need to require $ \eta $ to satisfy that
\begin{equation}\label{Formula:eta}
	0 < \frac{r_k^2}{1+2 r_k} = - \frac{ b^{(n)}_{1} }{ b^{(n)}_0 } < \eta < 1 \quad \text { for } k \geq 2,
\end{equation}
which can be achieved by setting $ 0 < r_k < 1 + \sqrt{2} $.

Due to \eqref{Formula:new_BDF2} and $ \psi^n = \phi^n - \eta \phi^{n-1} $, the first equation of \eqref{sch:2_1} also reads as
\begin{equation}\label{MBP:f1}
		\mg_h^{n} \phi^{n}  = \eta d^{(n)}_{0} \phi^{n-1} + \sum^{n-1}_{k=0} ( d^{(n)}_{n-k-1} - d^{(n)}_{n-k} ) \psi^{k} +  V( g_h(\hat{\phi}^{n}, R^{n-1}) )\, ( f(\hat{\phi}^{n}) + \kappa \hat{\phi}^{n} ) ,
\end{equation}
where $ \mg_h^{n} :=( d^{(n)}_{0} + \kappa ) I - \varepsilon^2 \Delta_h $, which will be used to evaluate $ \phi^n $ through the information of $ \phi^{n-1}$, $ \hat{\phi}^{n} $ and $ \{ \psi^{k} \}_{k=0}^{n-1} $. 
To evaluate $ \psi^n $, we substitute the updated formula \eqref{Formula:new_BDF2} and \eqref{new_ker:f1} into the first equation of \eqref{sch:2_1}  to derive
\begin{equation}\label{MBP:f2}
		\mg_h^{n} \psi^{n} 
		 = \sum^{n-1}_{k=0} Q^{n}_{n-k} \psi^{k} + V( g_h(\hat{\phi}^{n}, R^{n-1}))\, ( f(\hat{\phi}^{n}) + \kappa \hat{\phi}^{n} ),
\end{equation}
where
\begin{equation}\label{def:Q}
	Q^{n}_{k} := ( d^{(n)}_{k-1} - d^{(n)}_{k} - \kappa \eta^{k} )I + \eta^{k} \varepsilon^2 \Delta_h.
\end{equation}

\begin{lemma}\label{lem:boundQ} 
	Assume that the stepsize ratio $ 0 < r_k < 1 + \sqrt{2} $ for $2 \leq k \leq n$ holds, and the temporal stepsize satisfies
	\begin{equation}\label{Condition:tau1}
		\tau_n \leq \frac{\left(1+2 r_n\right) \eta-r_n^2}{\eta^2\left(1+r_n\right)} \frac{1-\eta}{ \kappa + 4 \varepsilon^2 h^{-2}} \quad {\rm for } \  n \geq 1.
	\end{equation}
	Then, for any constant $ 0 \leq \lambda \leq 1 $, we have
	\begin{equation}\label{boundQ_tilde}
		\big\| ( d^{(n)}_{k-1} - d^{(n)}_{k} - \kappa \lambda \eta^{k} )I + \eta^{k} \varepsilon^2 \Delta_h \big\|_{\infty} \leq d_{k-1}^{(n)}-d_k^{(n)} - \kappa \lambda \eta^k \quad {\rm for} \  1 \leq k \leq n.
	\end{equation}
	Specifically, if $ \lambda = 1 $, the matrix $Q_k^{n}$ defined in \eqref{def:Q} fulfills
	\begin{equation}\label{boundQ}
		\left\|Q_k^{n}\right\|_{\infty} \leq d_{k-1}^{(n)}-d_k^{(n)} - \kappa \eta^k \quad {\rm for } \  1 \leq k \leq n .
	\end{equation}
\end{lemma}
\begin{proof}
Note that \eqref{boundQ} is an apparent consequent of \eqref{boundQ_tilde}. For $ 0 \leq \lambda \leq 1 $, by considering $ \kappa \lambda $ as $ S_n $ in \cite[Lemma 4.1]{SINUM_2020_Liao} and following the same proof, we can directly give the estimation \eqref{boundQ_tilde}
in the maximum-norm under the following condition
\begin{equation}
\tau_n \leq \frac{\left(1+2 r_n\right) \eta-r_n^2}{\eta^2\left(1+r_n\right)} \frac{1-\eta}{ \kappa \lambda + 4 \varepsilon^2 h^{-2}},
\end{equation}
which further implies \eqref{Condition:tau1}.
\end{proof}

\begin{remark}\label{rem:eta} 
	Note that the restriction \eqref{Condition:tau1} on temporal stepsize is consistent with the previous conclusions yielded by the fully-implicit scheme \cite{SINUM_2020_Liao} and the implicit-explicit scheme \cite{MOC_2023_Ju}, which does not offer an explicit guideline for identifying the feasible range of the temporal stepsize. By using \eqref{Formula:eta}, it is easy to check that the right side of  \eqref{Condition:tau1} is strictly positive. For a fixed maximum stepsize ratio $ r_{\max} \in [1, 1 + \sqrt{2}) $, it is recommended by \cite{SINUM_2020_Liao} that one can choose the recombined parameter 
		$\eta_{*} := \frac{2 r_{\max}^2}{\left(1+r_{\max}\right)^2}$
	to relieve the restriction  \eqref{Condition:tau1}.
	In what follows, by default $ \eta = \eta_{*} $ is always set. Moreover, we set
		$$
		\mathcal{K}( s ) := \frac{1-\eta_{*}}{ \eta_{*}^2 } \frac{\left(1+2 s\right) \eta_{*}-s^2}{ 1+s }, \quad s \in ( 0, r_{\max} ).
		$$
	It can be verified that $ \mathcal{K}( s ) $ is increasing in $ ( 0, \sqrt{ 1 + \eta_{*} } - 1 ) $ and decreasing in $ ( \sqrt{ 1 + \eta_{*} } - 1, r_{\max} ) $; see \cite{MOC_2023_Ju} for more details. Moreover, since
	$$
	\mathcal{K}( 0 ) = \frac{ 1 - \eta_{*} }{ \eta_{*} } > \frac{ ( 1 - \eta_{*} ) ( 3 \eta_{*} - 1 ) }{ 2 \eta_{*}^{2} } = \mathcal{K}( 1 ) \geq \mathcal{K}( r_{\max} ),
	$$
	we have $ \mathcal{K}( r_{\max} ) \leq \mathcal{K}( r_{n} ) $ for all $ r_{n} \in ( 0, r_{\max} ) $. Thus, in practice, the time-step condition \eqref{Condition:tau1} can read as $
	\tau_n \leq \frac{ \mathcal{K}( r_{\max} ) }{ \kappa + 4 \varepsilon^2 h^{-2}}$.
	For instance, in the case of uniform meshes with $ r_{n} \equiv 1 $, one can take the recombined parameter $ \eta^{*} = \frac{1}{2} $, and the condition \eqref{Condition:tau1} now reduces to $ \tau_{n} = \tau \leq \frac{ 1 }{ 2 ( \kappa + 4 \varepsilon^2 h^{-2} ) } $.
\end{remark}

\begin{remark}\label{rem:Neumann_3D} 
	For the three-dimensional case, the discrete Laplace operator $\Delta_h$ is defined as
	$$
	\Delta_h v_{i j k}=\frac{1}{h^2}\left(v_{i+1, j, k}+v_{i-1, j, k}+v_{i, j+1, k}+v_{i, j-1, k} + v_{i, j, k+1} + v_{i, j, k-1} - 6 v_{i j k}\right)
	$$
	for $ 1 \leq i, j, k \leq M $. To ensure the non-negativity of the diagonal elements of the matrix $ Q^{n}_{k} $, i.e.,
	$$
	d^{(n)}_{k-1} - d^{(n)}_{k} - \kappa \eta^{k} \geq \frac{ 6 \eta^{k} \varepsilon^2 }{ h^{2} },
	$$
	the following time-step condition 
	$$
	\tau_n \leq \frac{\left(1+2 r_n\right) \eta-r_n^2}{\eta^2\left(1+r_n\right)} \frac{1-\eta}{ \kappa + 6 \varepsilon^2 h^{-2}}
	$$
	is required. In this sense, the conclusion \eqref{boundQ} can be established following the same argument as \cite[Lemma 4.1]{SINUM_2020_Liao}. Furthermore, similar results can be derived under homogeneous Neumann boundary condition using the same proof technique.
\end{remark}

Now, we are ready to establish the discrete MBP for the proposed variable-step BDF2-sESAV-I scheme in the following theorem.

\begin{theorem}[MBP of the BDF2-sESAV-I scheme]\label{thm:MBP_2} 
	Assume the conditions in Lemma \ref{lem:boundQ} and the boundedness assumption (\textbf{A2}) hold, then the variable-step BDF2-sESAV-I scheme \eqref{sch:2_1} preserves the MBP at the discrete level, i.e.,  
	$$
\text{if}~  \| \phi^0 \|_{\infty} \leq \beta \Longrightarrow 	\| \phi^n \|_{\infty} \leq \beta \quad {\rm for} \  1 \leq n \leq N.
	$$
\end{theorem}
\begin{proof}
The desired result will be proved together with the following conclusion
\begin{equation}\label{MBP:f3}
	\begin{aligned}
		\text{if}~  \| \phi^0 \|_{\infty} \leq \beta \Longrightarrow  \| \psi^{n} \|_{ \infty } \leq ( 1 - \eta ) \beta \quad \text{for} \  1 \leq n \leq N.
	\end{aligned}
\end{equation}
We start to prove this theorem with the mathematical induction argument. 

Due to $ \| \phi^0 \|_{\infty} \leq \beta $, it follows directly from Lemma \ref{lem:MBP_BDF1} that $ \| \hat{ \phi }^{1} \|_{\infty} \leq \beta $. Now taking $ n = 1 $ in \eqref{MBP:f1}, together with the facts $ d^{(1)}_{1} = \eta d^{(1)}_{0} $ and $ \psi^0 = \phi^0 $, one has
\begin{equation*}
	\begin{aligned}
		& \mg_h^{1} \phi^{1} = \eta d^{(1)}_{0} \phi^{0} + ( 1 - \eta) d^{(1)}_{0} \psi^{0} +  V( g_h(\hat{\phi}^{1}, R^{0}) )\, ( f(\hat{\phi}^{1}) + \kappa \hat{\phi}^{1} ) .
	\end{aligned}
\end{equation*}
Thus, the application of Lemmas \ref{lem:MBP_left}--\ref{lem:MBP_right} and assumption (\textbf{A2}) yields
$$
\begin{aligned}
	( d^{(1)}_{0} + \kappa ) \| \phi^{1} \|_{ \infty } & \leq d^{(1)}_{0} \| \phi^{0} \|_{ \infty } + \kappa V( g_h(\hat{\phi}^{1}, R^{0}) ) \| \hat{\phi}^{1} \|_{ \infty } \leq ( d^{(1)}_{0} + \kappa ) \beta,
\end{aligned}
$$
which implies $ \| \phi^{1} \|_{ \infty } \leq \beta $. 

Next, we shall bound $ \| \psi^{1} \|_{\infty} $ in the form of \eqref{MBP:f3}. Taking $ n = 1 $ in \eqref{MBP:f2} and using Lemmas \ref{lem:MBP_left}--\ref{lem:MBP_right} and \ref{lem:boundQ}, we get
\begin{equation*}
	\begin{aligned}
		( d^{(1)}_{0} + \kappa ) \| \psi^{1} \|_{\infty} & \leq  \| Q^{1}_{1} \|_{\infty} \| \psi^{0} \|_{\infty} + \kappa V( g_h(\hat{\phi}^{1}, R^{0}) ) \| \hat{\phi}^{1} \|_{\infty} \\
		& \leq ( d_{0}^{(1)}-d_1^{(1)} - \kappa \eta ) \beta + \kappa V( g_h(\hat{\phi}^{1}, R^{0}) ) \beta\\
		& = ( 1 - \eta ) ( d^{(1)}_{0} + \kappa ) \beta + \kappa \big( V( g_h(\hat{\phi}^{1}, R^{0}) ) - 1 \big) \beta.
	\end{aligned}
\end{equation*}
Thus \eqref{MBP:f3} holds for $ n = 1 $.
For any $ 2 \leq n \leq N $, we assume that
\begin{equation}\label{MBP:f4}
	\begin{aligned}
		\| \phi^{k} \|_{ \infty } \leq \beta \quad \text{and} \quad \| \psi^{k} \|_{ \infty } \leq ( 1 - \eta ) \beta \quad \text{for} \  1 \leq k \leq n-1.
	\end{aligned}
\end{equation}
Then, by Lemma \ref{lem:MBP_BDF1} we see $ \|\hat{\phi}^{n}\| \leq \beta $. This combined with \eqref{MBP:f1}, and using \eqref{Formula:new_ker_2} and  Lemmas \ref{lem:MBP_left}--\ref{lem:MBP_right} we have
\begin{equation}\label{MBP:f5}
	\begin{aligned}
		( d^{(n)}_{0} + \kappa ) \| \phi^n \|_{\infty} 
		& \leq \eta d^{(n)}_{0} \| \phi^{n-1} \|_{\infty} + \sum^{n-1}_{k=0} ( d^{(n)}_{n-k-1} - d^{(n)}_{n-k} ) \| \psi^{k} \|_{\infty}+ \kappa V( g_h(\hat{\phi}^{n}, R^{n-1}) ) \beta \\
		& \le \big( d^{(n)}_{0} + \kappa  \big) \beta + \kappa \big( V( g_h(\hat{\phi}^{n}, R^{n-1}) ) - 1 \big) \beta,
	\end{aligned}
\end{equation}
which together with assumption (\textbf{A2}) gives $ \| \phi^n \|_{\infty} \leq \beta $. 

It remains to evaluate $ \| \psi^{n} \|_{\infty} $. By applying Lemma \ref{lem:boundQ}, \eqref{Formula:new_ker_2} and the inductive hypothesis \eqref{MBP:f4}, equation \eqref{MBP:f2} gives rise to
\begin{equation}\label{MBP:f6}
	\begin{aligned}
		( d^{(n)}_{0} + \kappa ) \| \psi^n \|_{\infty}  & \leq \sum^{n-1}_{k=1} \| Q^{n}_{n-k} \|_{\infty} \| \psi^{k} \|_{\infty} + \| Q^{n}_{n} \|_{\infty} \| \psi^{0} \|_{\infty} +\kappa V( g_h(\hat{\phi}^{n}, R^{n-1}) ) \beta \\
		& \le (1 - \eta) ( d^{(n)}_{0} + \kappa ) \beta + \kappa \big( V( g_h(\hat{\phi}^{n}, R^{n-1}) ) - 1 \big) \beta,
	\end{aligned}
\end{equation}
which immediately leads to $ \| \psi^n \|_{\infty} \leq ( 1 - \eta ) \beta $ by assumption (\textbf{A2}), and thus the proof is completed.
\end{proof}

\begin{remark} 
	It is worth noting that, apart from the estimates of the matrices $ \mg_h^{n} $ and $ Q^{n}_{k} $ (see Lemmas \ref{lem:MBP_left} and \ref{lem:boundQ}, respectively), the proof of Theorem \ref{thm:MBP_2} does not rely on the space dimensions or boundary conditions. Moreover, it can be readily verified that under appropriate time-step restrictions, Lemmas \ref{lem:MBP_left} and \ref{lem:boundQ} remain valid for the three space dimensions and/or homogeneous Neumann boundary condition (see Remark \ref{rem:Neumann_3D}). Therefore, the extensions of the MBP-preserving result to these settings are straightforward.
\end{remark}

\begin{remark}\label{rem:MBP0} 
	Following the same proof of \eqref{MBP:f5}--\eqref{MBP:f6} and using assumption (\textbf{A2}), we can also derive
	\begin{equation}\label{MBP:sESAV2_f1}
		\begin{aligned}
			( d^{(n)}_{0} + \kappa ) \| \phi^n \|_{\infty} \leq ( d^{(n)}_{0} + \kappa ) \beta, \  ( d^{(n)}_{0} + \kappa ) \| \psi^n \|_{\infty}  \leq (1 - \eta) ( d^{(n)}_{0} + \kappa ) \beta 
		\end{aligned}
	\end{equation}
	for the BDF2-sESAV-II scheme, which also implies $ \| \phi^n \|_{\infty} \leq \beta $ and $ \| \psi^n \|_{\infty} \leq ( 1 - \eta ) \beta $. For the BDF2-sESAV-III scheme, we use the conclusion \eqref{boundQ_tilde} in Lemma \ref{lem:boundQ} with $ \lambda = V( g_h(\hat{\phi}^{n}, R^{n-1}) $ to get
	\begin{equation}\label{MBP:sESAV3_f1}
		\begin{aligned}
			( d^{(n)}_{0} + \kappa V( g_h(\hat{\phi}^{n}, R^{n-1}) ) ) \| \phi^n \|_{\infty} &\leq ( d^{(n)}_{0} + \kappa V( g_h(\hat{\phi}^{n}, R^{n-1}) ) ) \beta, \\
			( d^{(n)}_{0} + \kappa V( g_h(\hat{\phi}^{n}, R^{n-1}) ) ) \| \psi^n \|_{\infty}  &\leq (1 - \eta) ( d^{(n)}_{0} + \kappa V( g_h(\hat{\phi}^{n}, R^{n-1}) ) ) \beta,
		\end{aligned}
	\end{equation}
	which can also deduce the discrete MBP for the BDF2-sESAV-III scheme.
\end{remark}

\begin{remark}\label{rem:MBP} 
	Denote $ \mathcal{P}^n := \frac{ V( g_h(\hat{\phi}^{n}, R^{n-1}) ) - 1 }{ d^{(n)}_{0} + \kappa } \kappa \beta $, the results of \eqref{MBP:f5}--\eqref{MBP:f6} directly lead to
	\begin{equation*}
			\| \phi^n \|_{\infty} \leq \beta + \mathcal{P}^n \quad  {\rm and} \quad   \| \psi^n \|_{\infty} \leq ( 1 - \eta ) \beta + \mathcal{P}^n.
	\end{equation*}
	For a large temporal stepsize, $ g_h(\hat{\phi}^{n}, R^{n-1}) $ and $ V( g_h(\hat{\phi}^{n}, R^{n-1}) ) $ may deviate significantly from $ 1 $. In this case, the negative term $ \mathcal{P}^n $ can be considered as a penalty term to generate solutions with smaller $L^{\infty}$-norm. However, there is no such similar term in the BDF2-sESAV-II and BDF2-sESAV-III schemes, see \eqref{MBP:sESAV2_f1} and \eqref{MBP:sESAV3_f1}. This means that compared to these two schemes, the proposed BDF2-sESAV-I with a modified stabilization term will be more effective in preserving the MBP.
\end{remark}

\subsection{Discrete energy dissipation law}
Now, we consider the energy dissipation law of the variable-step BDF2-sESAV-I scheme \eqref{sch:2_1} by introducing a discrete modified energy as follows
\begin{equation}\label{def:dis_energy}
		\me_h[\phi^{n}, R^{n}] := \frac{\varepsilon^2}{2}\left\|\nabla_h \phi^n\right\|^2 + R^n + \langle G[ \nabla_{\tau} \phi^{n} ], 1 \rangle, ~n \ge 1,
\end{equation}
with $ \me_h[\phi^{0}, R^{0}] := \frac{\varepsilon^2}{2}\left\|\nabla_h \phi^0\right\|^2 + R^0 $  corresponding to the discrete version initial energy.

\begin{theorem}[Energy dissipation of the BDF2-sESAV-I scheme]\label{thm:energy_2} 
	Assume the conditions in Lemma \ref{lem:BDF2_P1} hold. Then for any $\kappa \geq 0$, the variable-step BDF2-sESAV-I scheme \eqref{sch:2_1} is energy dissipative in the sense that 
	\begin{equation}\label{dis_energy:e1}
			\me_h[\phi^{n}, R^{n}] \leq \me_h[\phi^{n-1}, R^{n-1}], ~n \ge 1.
	\end{equation}
\end{theorem}
\begin{proof}
Taking the inner product with the first equation of \eqref{sch:2_1} by $\nabla_{\tau} \phi^{n} = \phi^{n}-\phi^{n-1}$ yields
\begin{equation}\label{energy1_f1}
	\begin{aligned}
		& \big\langle \md_{2} \phi^{n}, \phi^{n}-\phi^{n-1} \big\rangle - \varepsilon^2 \big\langle\Delta_h \phi^{n}, \phi^{n}-\phi^{n-1}\big\rangle \\
		& \quad = \big\langle V ( g_h(\hat{\phi}^{n}, R^{n-1}) ) f(\hat{\phi}^{n}) - \kappa ( \phi^{n} - V ( g_h(\hat{\phi}^{n}, R^{n-1}) ) \hat{\phi}^{n} ), \phi^{n}-\phi^{n-1} \big\rangle.
	\end{aligned}
\end{equation}
It follows from Lemma \ref{lem:BDF2_P1} that $ \big\langle \md_{2} \phi^{1} , \phi^{1}-\phi^{0} \big\rangle \geq \langle G[ \nabla_{\tau} \phi^{1} ], 1 \rangle $ and
$$
\big\langle \md_{2} \phi^{n} , \phi^{n}-\phi^{n-1} \big\rangle \geq \langle G[ \nabla_{\tau} \phi^{n} ], 1 \rangle - \langle G[ \nabla_{\tau}  \phi^{n-1} ], 1 \rangle, \quad  n \geq 2,
$$
which together with the second equation of \eqref{sch:2_1}, \eqref{energy1_f1} and the identity
$$
	\big\langle\Delta_h \phi^{n}, \phi^{n}-\phi^{n-1}\big\rangle=  -\frac{1}{2}\left\|\nabla_h \phi^{n}\right\|^2 + \frac{1}{2}\left\|\nabla_h \phi^{n-1}\right\|^2 -\frac{1}{2}\left\|\nabla_h \phi^{n}-\nabla_h \phi^{n-1}\right\|^2
$$
gives us
$$
	\me_h[\phi^{n}, R^{n}]-\me_h[\phi^{n-1}, R^{n-1}] \leq  -\frac{\varepsilon^2}{2}\left\|\nabla_h \phi^{n}-\nabla_h \phi^{n-1}\right\|^2 \leq 0.
$$
The proof is completed.
\end{proof}

Like the previous SAV methods, Theorem \ref{thm:energy_2} only presents a modified energy dissipation law for the proposed BDF2-sESAV-I scheme. 
However, in our context,  we can show that the modified  energy is convergent to the original one with specified convergence orders. Meanwhile, the original energy dissipation law holds in an approximate setting. The proof is presented in \ref{App:A} following  error estimates for $R$ in Theorem \ref{thm:ErrR} and $\phi$ in Theorem \ref{thm:ErrPhi}.
	
\begin{theorem}[Energy approximations]\label{thm:Err_Energy}
		Assume the conditions in Theorems \ref{thm:MBP_2}--\ref{thm:energy_2} and the assumptions (\textbf{A1})--(\textbf{A2}) hold. Moreover, suppose that the exact solution $ \phi $ satisfies
		$
		\phi \in W^{3,\infty}(0,T;L^{\infty}(\Omega)) \cap W^{1,\infty}(0,T;W^{2,\infty}(\Omega)) \cap L^{\infty}(0,T;W^{4,\infty}(\Omega))
		$. If the temporal stepsize satisfies $ \tau_{1} \leq \tau^{4/3} $ and  $ \tau_{n} \leq \hat{ \tau } $,
		then we have
		\begin{equation}\label{thmErrE:0}
			\big\vert E_h[\phi^{n}] - \me_h[\phi^{n}, R^{n}] \big\vert =\mo ( \tau + h^2 ), 
		\end{equation}
		for $ n \geq 0$ and
		\begin{equation}\label{thmErrE:1}
			E_h[\phi^{n}] \leq E_h[\phi^{n-1}] +  \mo ( \tau + h^2 ), \quad E_h[\phi^{n}] \leq E_h[\phi^{0}] +  \mo ( \tau + h^2 ),
		\end{equation}
		for $ n \geq 1$.
		Here the discrete original energy is defined by $ E_h[\phi^{n}] := \frac{\varepsilon^2}{2}\left\|\nabla_h \phi^n\right\|^2 + E_{1h}[ \phi^{n} ] $.
	\end{theorem}
	\begin{remark}
		We remark that a related work \cite{MOC_2023_Ju} proposed a different linear BDF2 scheme and demonstrated its energy stability, showing that $E_{h}[ \phi^{n} ] \leq E_{h}[ \phi^{0} ] + C$ under certain stepsize conditions, specifically $ \max_{1\leq n \leq N} \tau_{n} / \min_{1\leq n \leq N} \tau_{n} \leq \Gamma$ with a finite constant $ \Gamma $ and $ h = O ( \sqrt{ \tau } ) $, see \cite[Remark 4.2]{MOC_2023_Ju}. In contrast, 	we demonstrate in this work an improved original energy stability \eqref{thmErrE:1} without any  stepsize limitation, see Theorem \ref{thm:Err_Energy}.
	\end{remark}

\section{Error analysis}\label{sec:Error}
In this section, we shall present error analysis for the variable-step BDF2-sESAV-I scheme \eqref{sch:2_1}. Define the error functions as $
\hat{e}_{\phi}^{n} := \phi( t_{n} ) - \hat{\phi}^n$, $e_{\phi}^{n} := \phi( t_{n} ) - \phi^n$ and  $e_{R}^{n} := R( t_{n} ) - R^n$. From the energy dissipation law, it holds that for all $ n \ge 1$,
\begin{equation*}\label{Bound_s}
	\| \nabla_{h} \phi^{n} \| + R^{n} \leq \mathcal{M}, \quad \text{if} \  \phi( \mathbf{x}, 0 ) \in H^{1} ( \Omega ),
\end{equation*}
where $ \mathcal{M} $ is a positive constant depending only on $ \phi_{\text{init}}( \mathbf{x} ) $ and the domain $ \Omega $. 
Moreover, we assume that there exists constants $ \overline{K} \geq \underline{K} \geq 0 $ such that
\begin{equation}\label{Bound_expf}
	\underline{K} \leq \| F( \zeta ) \|_{\infty}, \| f( \zeta ) \|_{\infty}, \| f^{\prime}( \zeta ) \|_{\infty} \leq \overline{K} \quad \text{for}~ \zeta \in [-\beta,\beta].
\end{equation}
In fact, the maximum bound principle, see \eqref{MBP}, Lemma \ref{lem:MBP_BDF1} and Theorem \ref{thm:MBP_2}, implies the assumption \eqref{Bound_expf} holds for
$ \zeta = \phi( t_{n} ) $, $ \hat{\phi}^n $ or $ \phi^{n} $.  

Next, we present an intermediate error estimate for $ \hat{e}_{\phi}^{n} $, which is related to $ e_{\phi}^{n-1}$ and will be used to derive the main error estimate for the variable-step BDF2-sESAV-I scheme. The detailed proof is presented in  \ref{App:B}.
\begin{lemma}[Estimate for $ \hat{e}_{\phi}^{n} $]\label{lem:estihat} 
	Assume the conditions in Theorems \ref{thm:MBP_2}--\ref{thm:energy_2} hold and suppose the exact solution $ \phi $ satisfies
	$
	\phi \in W^{2,\infty}(0,T;L^{\infty}(\Omega)) \cap L^{\infty}(0,T;W^{4,\infty}(\Omega))
	$, then we have
	\begin{equation*}
		\| \hat{e}_{\phi}^{n} \| \leq C ( \| e_{\phi}^{n-1} \| + \tau_{n} ( \tau_{n} + h^2 )), \quad \| \hat{e}_{\phi}^{n} \|_{\infty} \leq C (\| e_{\phi}^{n-1} \|_{\infty} + \tau_{n} ( \tau_{n} + h^2 )).
	\end{equation*}
\end{lemma}

Subtracting \eqref{sch:2_1} from \eqref{sav_1}, one gets the error equations
\begin{equation}\label{ErrR_3}
	\begin{aligned}
		\md_{2} e_\phi^{n} & = \varepsilon^2 \Delta_h e_\phi^{n} + f( \phi( t_{n} ) ) - V( g_h(\hat{\phi}^{n}, R^{n-1}) ) f(\hat{\phi}^{n})  \\
		            & \qquad + \kappa ( \phi^{n} - V ( g_h(\hat{\phi}^{n}, R^{n-1}) ) \hat{\phi}^{n} ) + T_2^n,
	\end{aligned}
\end{equation}
\begin{equation}\label{ErrR_4}
	\begin{aligned}
		\md_{1} e_R^{n} & = - \big\langle f(\phi(t_n)), \md_{1} \phi(t_n) \big\rangle + \big\langle V ( g_h(\hat{\phi}^{n}, R^{n-1}) ) f(\hat{\phi}^{n}), \md_{1} \phi^{n} \big\rangle \\
		& \quad - \kappa \big\langle  \phi^{n} - V ( g_h(\hat{\phi}^{n}, R^{n-1}) ) \hat{\phi}^{n} , \md_{1} \phi^{n} \big\rangle + T_{3}^n,
	\end{aligned}
\end{equation}
where the truncation errors $ T_2^n $ and $ T_3^n $ satisfy \cite{SINUM_2019_Chen,JSC_Qiao_2023,JSC_Ju_2022}:
\begin{equation}\label{ErrT_21}
		\begin{aligned}
	\| T_2^1 \|_{\infty} &=\mo(\tau_{1}  + h^2), ~	\| T_2^n \|_{\infty} =\mo( ( \tau_{n-1} + \tau_{n} )^2  + h^2 ), \  n \geq 2, \\
	\vert T_3^n \vert &=\mo( \tau_{n}   + h^2 ), \  n \geq 1.
		\end{aligned}
\end{equation}

The following theorem provides an error analysis for $  e_{R}^{n} $, along with the $H^{1}$-norm error estimate for $ e_{\phi}^{n} $, and also its detail proof is presented in \ref{App:C}.
\begin{theorem}[Error analysis for $  e_{R}^{n} $ and $ e_{\phi}^{n} $ in $H^{1}$-norm]\label{thm:ErrR} 
	Assume the conditions in Theorems \ref{thm:MBP_2}--\ref{thm:energy_2} and the assumptions (\textbf{A1})--(\textbf{A2}) hold. Moreover, suppose that the exact solution $ \phi $ satisfies
	$
	\phi \in W^{3,\infty}(0,T;L^{\infty}(\Omega)) \cap W^{1,\infty}(0,T;W^{2,\infty}(\Omega)) \cap L^{\infty}(0,T;W^{4,\infty}(\Omega))
	$. If the temporal stepsize satisfies $ \tau_{1} \leq \tau^{4/3} $ and $\tau_{n} \leq  \hat{ \tau }$,
	then it holds  
	\begin{equation}\label{thmErrR:1}
		\big\langle G[ \nabla_{\tau} e_\phi^{n} ], 1 \big\rangle + \frac{\kappa}{4} \| e_\phi^{n} \|^2 + \frac{\varepsilon^2 }{2} \| \nabla_{h} e_\phi^{n} \|^2 + \vert e_R^{n} \vert^2 =\mo ( \tau^2 + h^4 ).
	\end{equation}
	Moreover, we have
	\begin{equation}\label{thmErrR:2}
		\big\langle G[ \nabla_{\tau} e_\phi^{n} ], 1 \big\rangle  + \frac{\kappa}{4} \| e_\phi^{n} \|^2 + \frac{\varepsilon^2 }{2} \| \nabla_{h} e_\phi^{n} \|^2 =\mo ( \tau^4 + h^4 ).
	\end{equation}
\end{theorem}

\begin{remark}
	Note that in the proof of Theorem \ref{thm:ErrR}, only the estimate of $ \| T_{2}^{n} \| $ is utilized; see \eqref{ErrR_15}--\eqref{ErrR_15_2}. Therefore, by applying the well-known Bramble-Hilbert Lemma \cite{NM_1971_Hilbert}, the spatial regularity assumption in Theorem \ref{thm:ErrR} can be relaxed from $ W^{4,\infty} ( \Omega ) $ to $ H^{4}(\Omega) $.
\end{remark}

In Theorem \ref{thm:ErrR}, we have derived an $ H^{1} $-norm  error estimate for $ e_{\phi}^{n} $. However, this cannot assure an optimal $L^{\infty}$-norm error estimate for the Allen--Cahn equation \eqref{Model:tAC} in multi-dimensional spaces. An important lemma for the discrete $ L^{\infty}$-norm error analysis is listed below, which can be proved by using the technique of the discrete complementary convolution kernels, see Refs. \cite{MOC_2023_Ju,SINUM_2020_Liao}.

\begin{lemma}[\cite{MOC_2023_Ju}]\label{lem:Gron}
	Let $\zeta>0$ and $\lambda \in(0,1)$ be two constants. For any non-negative sequences $\{v^k\}_{k=1}^N$ and $\{w^k\}_{k=0}^N$ such that
	$$
	\sum_{k=1}^n d_{n-k}^{(n)} \nabla_\tau w^k \leq \zeta \sum_{k=1}^{n-1} \lambda^{n-k-1} w^k + v^n \quad {\rm for } \  1 \leq n \leq N,
	$$
	where the discrete kernels $\{d_k^{(n)}\}$ are defined in \eqref{Formula:new_ker_1}, it holds
	$$
	w^n \leq \exp \Big(\frac{\zeta t_n}{1-\lambda}\Big)\Big(w^0+\sum_{k=1}^n \frac{v^k}{b_0^{(k)}}\Big) \quad {\rm for } \  1 \leq n \leq N .
	$$
\end{lemma}

\begin{theorem}[Error analysis for $ e_{\phi}^{n} $ in the $ L^{\infty} $-norm]\label{thm:ErrPhi} 
	Assume the conditions in Theorem \ref{thm:ErrR} hold, 
	then we have
	\begin{equation}\label{thmErrPhi:1}
		\| e_\phi^{n} \|_{\infty} =\mo ( \tau^2 + h^2 ).
	\end{equation}
\end{theorem}
\begin{proof}
As before, we define $ e_{\psi}^k := e_{\phi}^k - \eta e_{\phi}^{k-1} $ for $ k \geq 1 $ with $ e_{\psi}^0 := e_{\phi}^0 = 0$. It is easy to check that \eqref{new_ker:f1}--\eqref{Formula:new_BDF2} also hold for $ e_{\phi}^k $ and $ e_{\psi}^k $. Note that
$$
\phi^{n} - V ( g_h(\hat{\phi}^{n}, R^{n-1}) ) \hat{\phi}^{n} = \hat{e}_{\phi}^n - e_{\phi}^n + ( 1 - V ( g_h(\hat{\phi}^{n}, R^{n-1}) ) ) \hat{\phi}^{n},
$$
the error equation \eqref{ErrR_3} can be rewritten as
\begin{equation}\label{ErrPhi_2}
	\begin{aligned}
		\mg_h^{n} e_{\psi}^n  & =\sum^{n-1}_{k=1} Q^{(n)}_{n-k} e_{\psi}^k + f( \phi( t_{n} ) ) - V( g_h(\hat{\phi}^{n}, R^{n-1}) ) f(\hat{\phi}^{n}) + \kappa \hat{e}_{\phi}^n\\
		& \qquad   + \kappa ( 1 - V ( g_h(\hat{\phi}^{n}, R^{n-1}) ) ) \hat{\phi}^{n} + T_2^n.
	\end{aligned}
\end{equation}

It follows from Lemma \ref{lem:estihat} and Theorem \ref{thm:ErrR} that
$ \| \hat{e}_{\phi}^{n} \|  =\mo ( \tau^2 + h^2 )$, which together with Lemma \ref{lem:estig}  implies
\begin{equation*}
\begin{aligned}
	\big\vert V( g_h(\hat{\phi}^{n}, R^{n-1}) ) - 1 \big\vert \leq C( \| \hat{e}_\phi^{n} \| + \vert e_{R}^{n-1} \vert + \tau )^2 =\mo ( \tau^2 + h^4 ),
\end{aligned}
\end{equation*}
and thus
\begin{equation*}
		\| f( \phi( t_{n} ) ) - V( g_h(\hat{\phi}^{n}, R^{n-1}) ) f(\hat{\phi}^{n}) \|_{\infty} \leq \overline{K} \| \hat{e}_\phi^{n} \|_{\infty} + \mo (  \tau^2 + h^4 ).
\end{equation*}
Then, by applying Lemmas \ref{lem:MBP_left}, \ref{lem:MBP_BDF1}, \ref{lem:estihat} and the above estimates to \eqref{ErrPhi_2}, one obtain
\begin{equation*}
	\begin{aligned}
		d^{(n)}_{0} \| e_{\psi}^n \|_{\infty} 
		& \leq  \sum^{n-1}_{k=1} \| Q^{(n)}_{n-k} \|_{\infty} \| e_{\psi}^k \|_{\infty} + \| f( \phi( t_{n} ) ) - V( g_h(\hat{\phi}^{n}, R^{n-1}) ) f(\hat{\phi}^{n}) \|_{\infty} \\
		& \quad  + \kappa \| \hat{e}_\phi^{n} \|_{\infty} + \kappa \big\vert V( g_h(\hat{\phi}^{n}, R^{n-1}) ) - 1 \big\vert \beta + \| T_2^n \|_{\infty} \\
		& \leq \sum^{n-1}_{k=1} \| Q^{(n)}_{n-k}\|_{\infty} \| e_{\psi}^k \|_{\infty} + C ( \| e_\phi^{n-1} \|_{\infty} + \| T_2^n \|_{\infty} + \tau^2 + h^2 ).
	\end{aligned}
\end{equation*}
Then, under the temporal stepsize restriction in Lemma \ref{lem:boundQ}, we see
\begin{equation*}
	    	d^{(n)}_{0} \| e_{\psi}^n \|_{\infty} 
		\leq \sum^{n-1}_{k=1} \big( d_{n-k-1}^{(n)}-d_{n-k}^{(n)} - \kappa \eta^{n-k} \big) 
	    	\| e_{\psi}^k \|_{\infty} 
		+ C \big( \| e_\phi^{n-1} \|_{\infty} + \| T_2^n \|_{\infty} + \tau^2 + h^2 \big),
\end{equation*}
which leads to
\begin{equation*}
	\begin{aligned}
		\sum^n_{k=1} d_{n-k}^{(n)} \nabla_{\tau} \| e_{\psi}^k \|_{\infty} \leq C \big( \| e_\phi^{n-1} \|_{\infty} + \| T_2^n \|_{\infty} + \tau^2 + h^2 \big).
	\end{aligned}
\end{equation*}
And then, using the relationship between $e_\phi$ and $e_\psi$ in formula \eqref{new_ker:f1},  we have
\begin{equation*}\label{ErrPhi_3}
	\begin{aligned}
		\sum^n_{k=1} d_{n-k}^{(n)} \nabla_{\tau} \| e_{\psi}^k \|_{\infty} \leq C \sum^{n-1}_{k=1} \eta^{n-k-1} \| e_\psi^{k} \|_{\infty} + C ( \| T_2^n \|_{\infty} + \tau^2 + h^2 ).
	\end{aligned}
\end{equation*}
Thus, an application of  Lemma \ref{lem:Gron} leads to
\begin{equation*}
	\begin{aligned}
		\| e_{\psi}^n \|_{\infty} & \leq  \exp \Big(\frac{C t_n}{1-\eta}\Big) \sum_{k=1}^n \frac{1}{b_0^{(k)}} \left( \| T_2^k \|_{\infty} + \tau^2 + h^2  \right) \\
		& \leq C \Big( \tau_1 ( \tau + h^2  ) + \sum_{k=2}^n \tau_k ( \tau^2 + h^2  ) \Big) =\mo ( \tau^2 + h^2  ) ,  
	\end{aligned}
\end{equation*}
where \eqref{Def:DCK} together with the truncation errors in \eqref{ErrT_21}  are applied.

Finally, we use the substitution formula \eqref{new_ker:f1} again to see
\begin{equation*}\label{ErrPhi_4}
	\| e_{\phi}^n \|_{\infty} \leq \sum_{k=1}^{n} \eta^{n-k} \| e_\psi^k \|_{\infty} \leq C  ( \tau^2 + h^2  ) \sum^{ n }_{k=1} \eta^{n-k} =\mo( \tau^2 + h^2  ),
\end{equation*}
for $ 0 < \eta < 1 $. Thus, the conclusion \eqref{thmErrPhi:1} is proved.
\end{proof}

\section{Numerical experiments}\label{sec:NumTest}
This section is devoted to numerical tests of the variable-step BDF2-sESAV-I scheme \eqref{sch:2_1}. There are two types of commonly-used nonlinear functions $ f(\phi) $. One is determined by the \emph{double-well potential} \cite{SISC_2019_Akrivis}
\begin{equation}\label{poten:dw}
	F( \phi ) = \frac{1}{4} ( \phi^2 - 1 )^2,
\end{equation}
which gives $
f(\phi) = -F'( \phi ) = \phi - \phi^3 $. In this case, one has $ \beta = 1 $ and $ \| f' \|_{ C[-\beta,\beta] } = 2 $. The other one is given by the \emph{Flory--Huggins potential} \cite{NM_1992_Elliott}
\begin{equation}\label{poten:fh}
	F( \phi ) = \frac{\theta}{2} [ ( 1 + \phi ) \ln ( 1 + \phi ) + ( 1 - \phi ) \ln ( 1 - \phi ) ] - \frac{ \theta_{c} }{2} \phi^2
\end{equation}
with $ \theta_{c} > \theta > 0 $, i.e., $ f(\phi) = - F'( \phi ) = \frac{\theta}{2} \ln \frac{ 1 - \phi }{ 1 + \phi } + \theta_{c} \phi $. In the following tests, we set $ \theta = 0.8 $ and $ \theta_{c} = 1.6 $, which gives us $ \beta \approx 0.9575 $ and $ \| f' \|_{ C[-\beta,\beta] } \approx 8.02 $ \cite{SINUM_Ju_2022,JSC_Ju_2022}. In our modeling, the auxiliary piecewise polynomial funcitional \eqref{f_cut} is adopted as $V(\cdot)$ and the stabilization parameter $ \kappa = \| f' \|_{ C[-\beta,\beta] } $ is always set for both cases.

\subsection{Temporal convergence}
We test the temporal accuracy of the variable-step BDF2-sESAV-I scheme by setting $ \varepsilon^2 = 0.01 $ in \eqref{Model:tAC} and considering a smooth initial value
$
\phi_{\text {init}} ( x, y ) = 0.1 \sin x \sin y$  in $\Omega = (0,2\pi)^2.
$
Let the terminal time $ T = 1 $ and choose $ M = 256 $. Since there is no analytical solution available for this problem, we use the numerical solution yielded by the BDF2-SAV scheme \cite{JSC_Qiao_2023} with uniform temporal stepsize $ \tau = 2^{-12} $ as the reference solution. In this example, the random temporal grids are generated randomly by
$$
\tau_{k} := T \frac{\theta_{k}}{S}, \quad {\rm with } \  S = \sum^{N}_{k=1} \theta_{k}, 
$$
where $ \theta_{k} $ is randomly drawn from the uniform distribution on the interval $ ( 1/r_{\max}, 1 ) $ such that the adjacent temporal stepsize ratio $ r_{k} < r_{\max} = 2.4 $. For the double-well potential, the $L^{\infty}$-norm error of $ \phi $ and the error of $g_h(\phi,R)$ yielded by the BDF2-sESAV-I scheme with different time-stepping strategies, i.e., the uniform and random temporal stepsize, are presented in Fig. \ref{figEx1_1}. It is shown that the expected second- and first-order convergence rates are respectively achieved in all cases. Corresponding numerical results for the Flory--Huggins potential are plotted in Fig. \ref{figEx1_2}, and similar conclusions can also be observed.
\begin{figure}[!h]
	\vspace{-12pt}
	\centering
	\subfigure[error of $\phi$]
	{
		\includegraphics[width=0.45\textwidth]{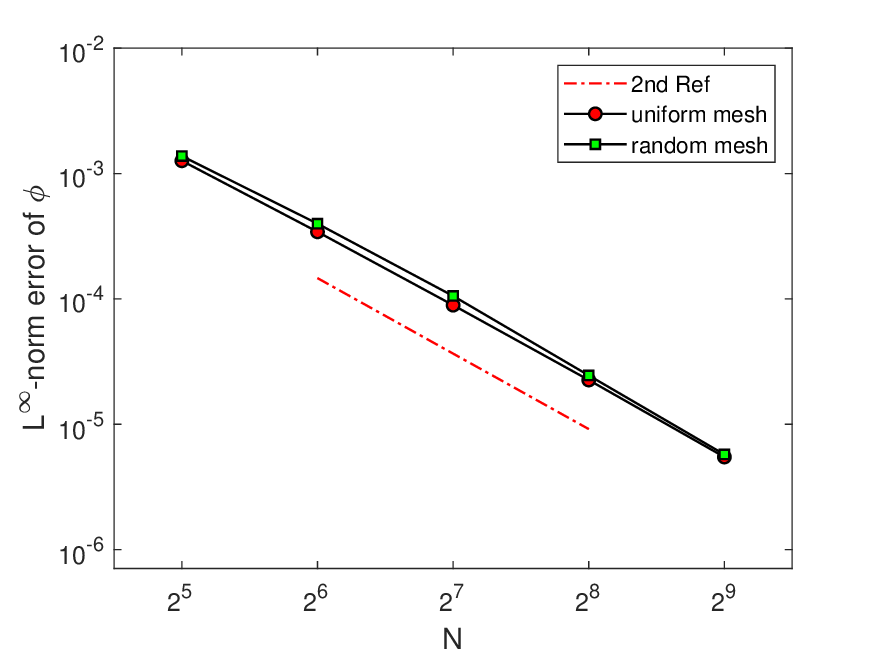}		\label{figEx1_1a}
	}%
	\subfigure[error of $g_h(\phi,R)$]
	{
		\includegraphics[width=0.45\textwidth]{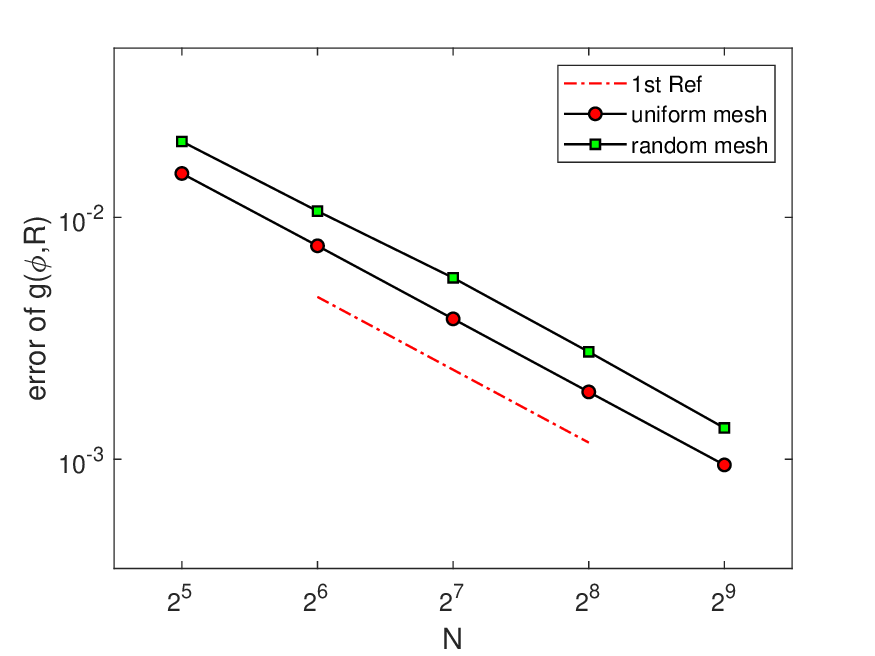}	\label{figEx1_1b}
	}%
	\setlength{\abovecaptionskip}{0.0cm} 
	\setlength{\belowcaptionskip}{0.0cm}
	\caption{The $L^{\infty}$-norm error of $ \phi $ (left) and the error of $g_h(\phi,R)$ (right)  generated by the (variable-step) BDF2-sESAV-I scheme: the double-well potential}	
	\label{figEx1_1}
\end{figure}
\begin{figure}[!h]
	\vspace{-12pt}
	\centering
	\subfigure[error of $\phi$]
	{
		\includegraphics[width=0.45\textwidth]{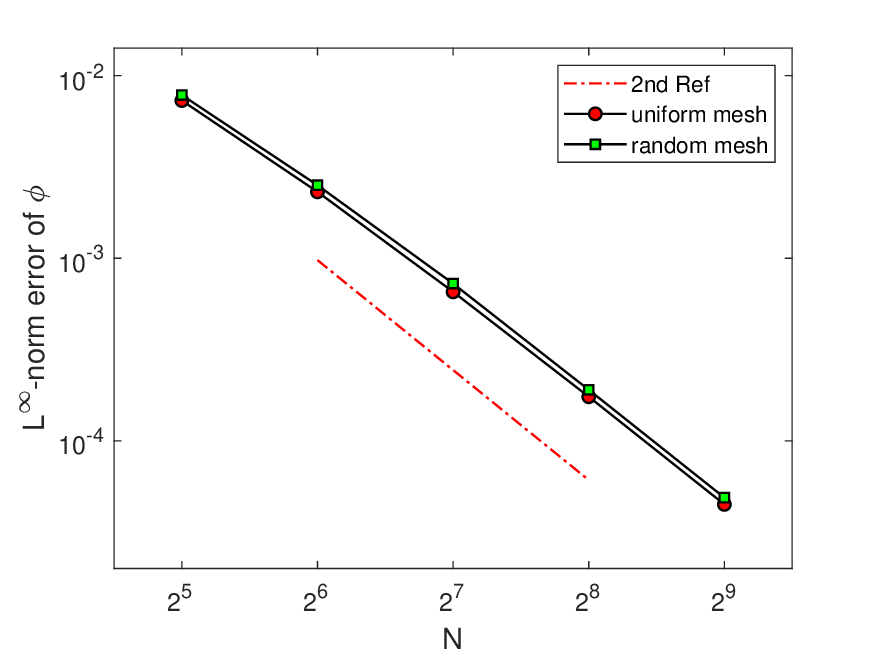}		\label{figEx1_2a}
	}%
	\subfigure[error of $g_h(\phi,R)$]
	{
		\includegraphics[width=0.45\textwidth]{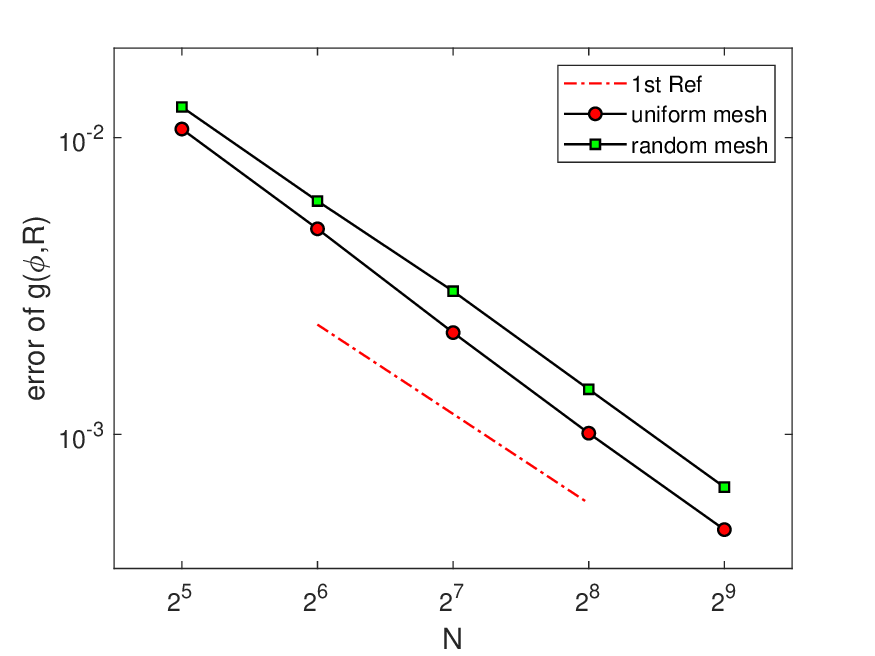}	\label{figEx1_2b}
	}%
	\setlength{\abovecaptionskip}{0.0cm} 
	\setlength{\belowcaptionskip}{0.0cm}
	\caption{The $L^{\infty}$-norm error of $ \phi $ (left) and the error of $g_h(\phi,R)$ (right) generated by the (variable-step) BDF2-sESAV-I scheme: the Flory--Huggins potential}	
	\label{figEx1_2}
\end{figure}

\subsection{Preservation of energy dissipation law and MBP}\label{sec:ex2}
In this subsection, we numerically verify the energy dissipation law and MBP of the proposed BDF2-sESAV scheme on both uniform and non-uniform temporal grids. Consider the grain coarsening dynamics governed by the Allen--Cahn equation \eqref{Model:tAC} with $ \varepsilon = 0.01 $ in $ \Omega=(0,1)^2 $. In this simulation, we take $ M=128 $ and the initial phase field is randomly generated between $ [-0.8,0.8] $, which is highly oscillating. 

\begin{figure}[!htbp]
	\vspace{-12pt}
	\centering
	\subfigure[fixed $ \tau = 0.4 $]
	{
		\includegraphics[width=0.45\textwidth]{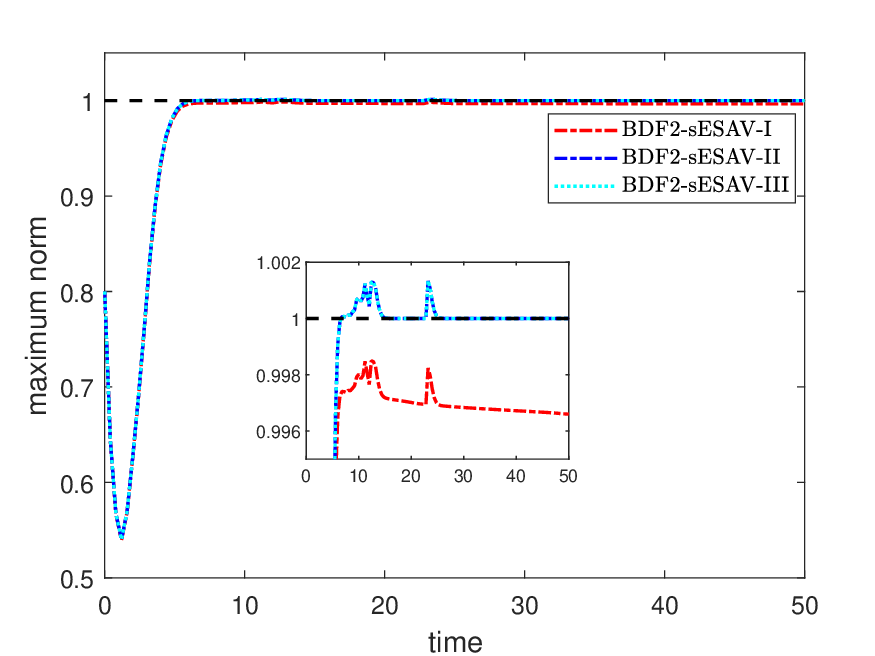}		\label{figEx2_1a}
	}%
	\subfigure[fixed $ \tau = 0.04 $]
	{
		\includegraphics[width=0.45\textwidth]{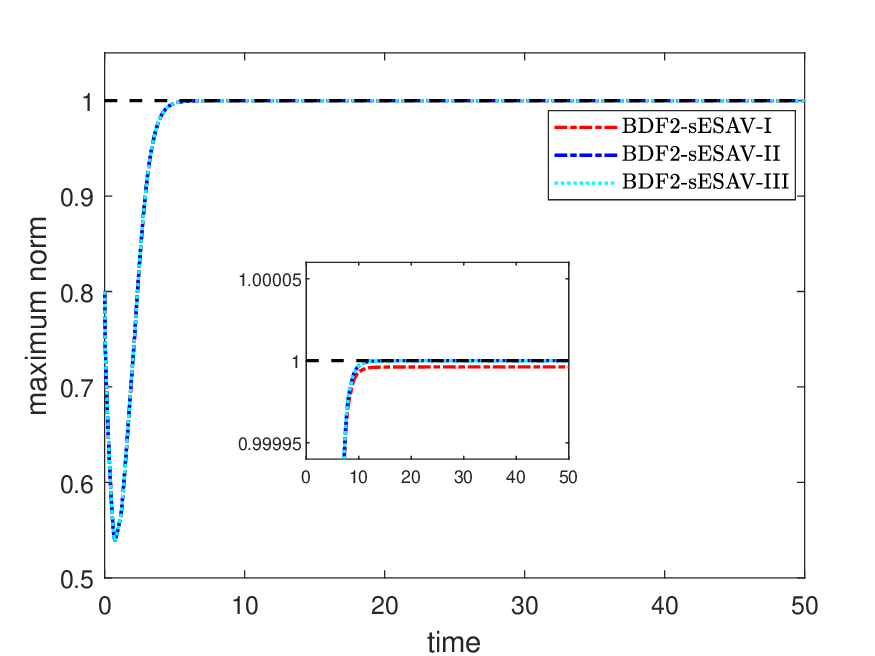}	\label{figEx2_1b}
	}%
	\setlength{\abovecaptionskip}{0.0cm} 
	\setlength{\belowcaptionskip}{0.0cm}
	\caption{Evolution of  simulated  solutions in maximum-norm computed by the BDF2-sESAV schemes with different stabilization terms: the double-well potential}	
	\label{figEx2_1}
\end{figure}
\begin{figure}[!ht]
	\vspace{-12pt}
	\centering
	\subfigure[fixed $ \tau = 0.1 $]
	{
		\includegraphics[width=0.45\textwidth]{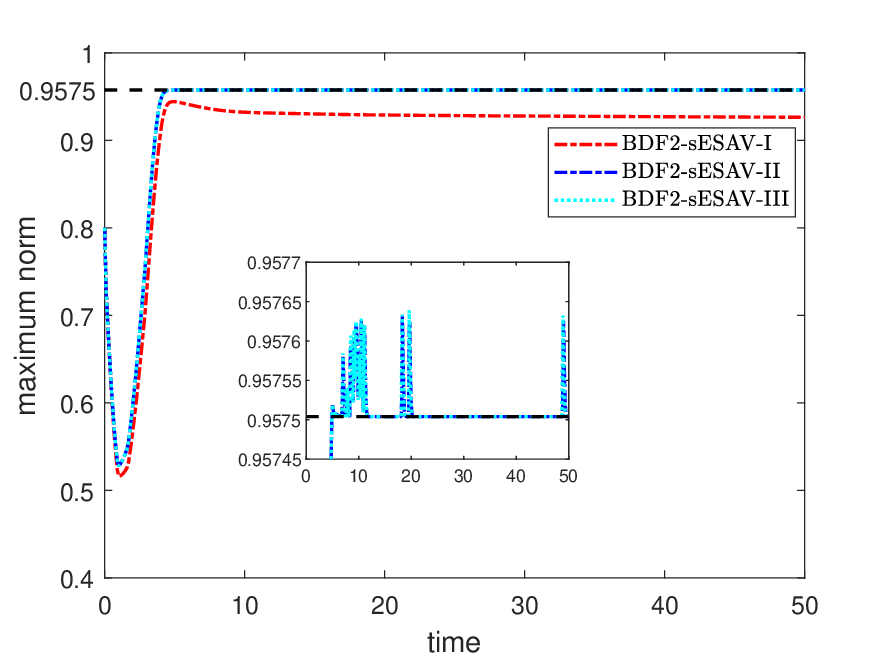}		\label{figEx2_2a}
	}%
	\subfigure[fixed $ \tau = 0.01 $]
	{
		\includegraphics[width=0.45\textwidth]{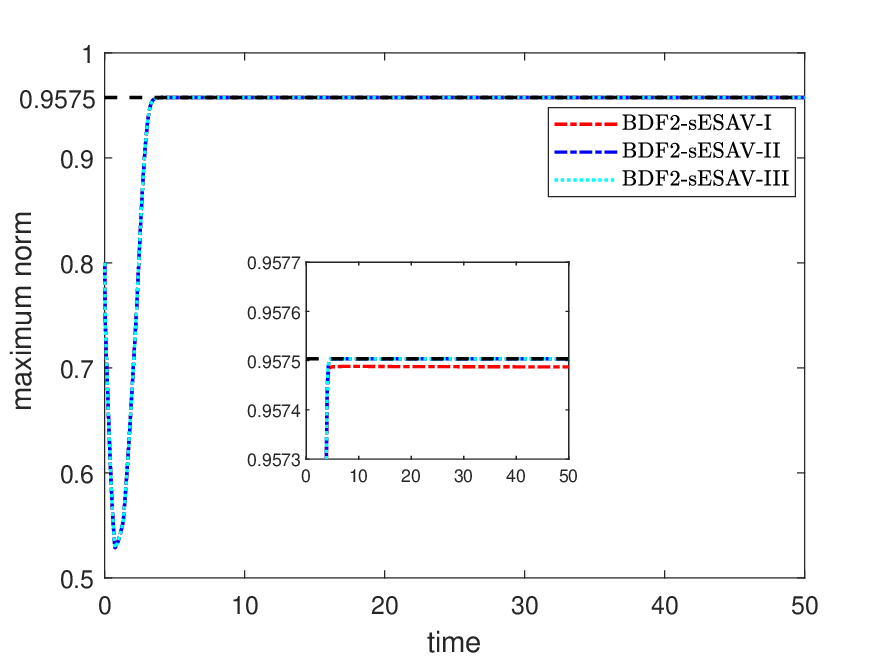}	\label{figEx2_2b}
	}%
	
	\setlength{\abovecaptionskip}{0.0cm} 
	\setlength{\belowcaptionskip}{0.0cm}
	\caption{Evolution of simulated solutions in maximum-norm  computed by the BDF2-sESAV schemes with different stabilization terms: the Flory--Huggins potential}	
	\label{figEx2_2}
	\vspace{-10pt}	
\end{figure}
To preserve the discrete MBP with uniform temporal grids, the temporal stepsize $ \tau \le 0.0585 $ for the double-well potential and $ \tau \le 0.0343 $ for the Flory--Huggins potential are required according to \eqref{Condition:tau1}. These conditions are also consistent with those of the fully-implicit \cite{SINUM_2020_Liao} and implicit-explicit \cite{MOC_2023_Ju} BDF2 schemes. As \eqref{Condition:tau1} is only a sufficient condition, therefore, in the following tests, we employ two different types of temporal girds (one satisfies the condition, while the other violates it) for each potential, i.e., $ \tau = 0.4 $, $ 0.04 $ for the double-well case and $ \tau = 0.1 $, $ 0.01 $ for the Flory--Huggins case. The coarsening dynamics using the BDF2-sESAV schemes with different stabilization terms (cf. \eqref{def:ST} and Remark \ref{def:other_S}) are tested. Fig. \ref{figEx2_1} displays the evolution of  simulated  solutions in maximum-norm  yielded by these three schemes for the double-well potential. It shows that the BDF2-sESAV-I scheme with the novel developed stabiliztion term can well preserve the MBP for both temporal grids, whereas the solutions of the BDF2-sESAV-II and BDF2-sESAV-III schemes obviously exceed the maximum bound $ \beta=1 $ when using the larger stepsize $ \tau = 0.4 $. This demonstrates the effectiveness and robustness of the proposed BDF2-sESAV-I scheme. Similar conclusions can also be drawn for the Flory--Huggins potential, as shown in Fig. \ref{figEx2_2}.

We next investigate the discrete energy dissipation law of the proposed BDF2-sESAV-I scheme with uniform temporal grids until $ T = 50 $, as well as the approximation of  the original energy by the modified energy. Numerical results for the double-well potential and the Flory--Huggins potential are depicted in Fig. \ref{figEx2_3} and Fig. \ref{figEx2_4}, respectively. As illustrated, both the original  and modified discrete energies decrease as time marching, and moreover, the modified energy approaches the original one as the temporal stepsize is reduced. Notably, the approximation rate of the modified energy to the original one is demonstrated to be first-order in time, as shown in Fig. \ref{figEx2_3c} and Fig. \ref{figEx2_4c}, which supports the theoretical results presented in Theorem \ref{thm:Err_Energy}.
\begin{figure}[!ht]
	\vspace{-8pt}
	\centering
	\subfigure[$\tau = 0.4$]
	{
		\includegraphics[width=0.33\textwidth]{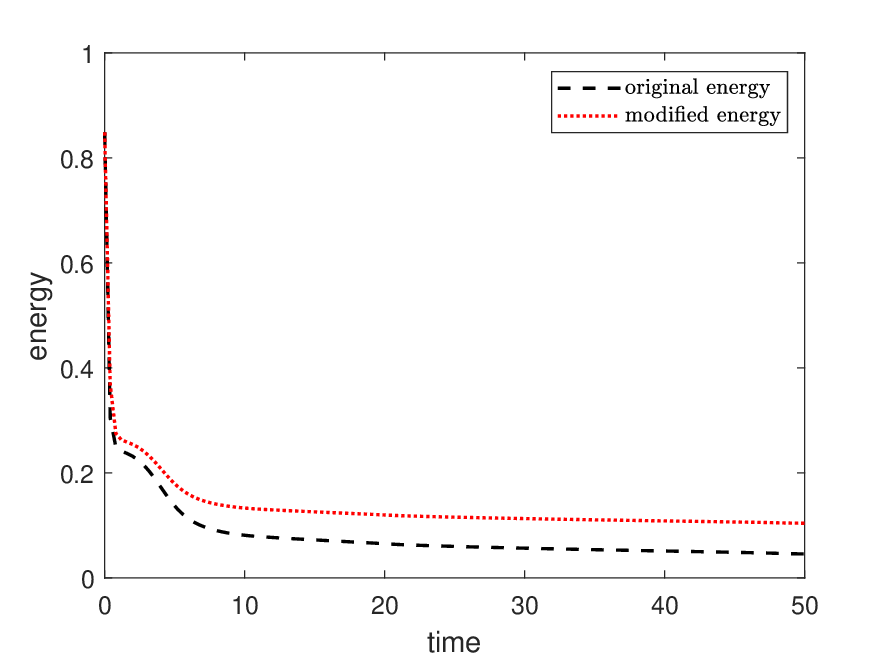}
		\label{figEx2_3a}
	}%
	\subfigure[$\tau = 0.04$]
	{
		\includegraphics[width=0.33\textwidth]{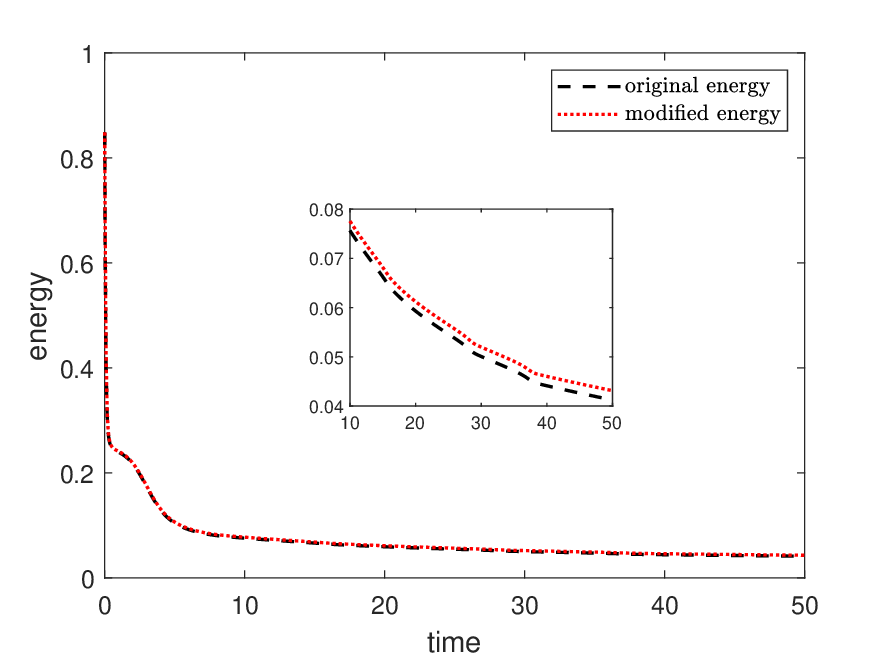}
		\label{figEx2_3b}
	}%
	\subfigure[$ \vert E_{h}{[\phi]} - \me_{h}{[\phi,R]} \vert $]
	{
		\includegraphics[width=0.33\textwidth]{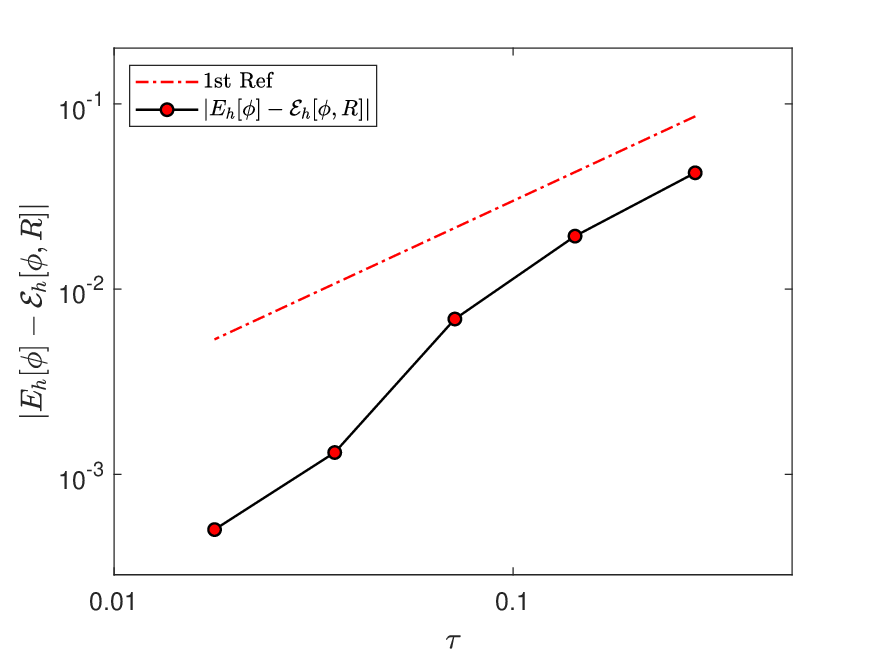}
		\label{figEx2_3c}
	}%
	\setlength{\abovecaptionskip}{0.0cm} 
	\setlength{\belowcaptionskip}{0.0cm}
	\caption{Evolutions of the original and modified energies with $\tau = 0.4$ (left) and $ \tau = 0.04 $ (middle), and the approximation rate of the modified energy to the original one (right) for the BDF2-sESAV-I scheme: the double-well potential}
	\label{figEx2_3}
\end{figure}
\begin{figure}[!ht]
	\vspace{-10pt}
	\centering
	\subfigure[$\tau = 0.1$]
	{
		\includegraphics[width=0.33\textwidth]{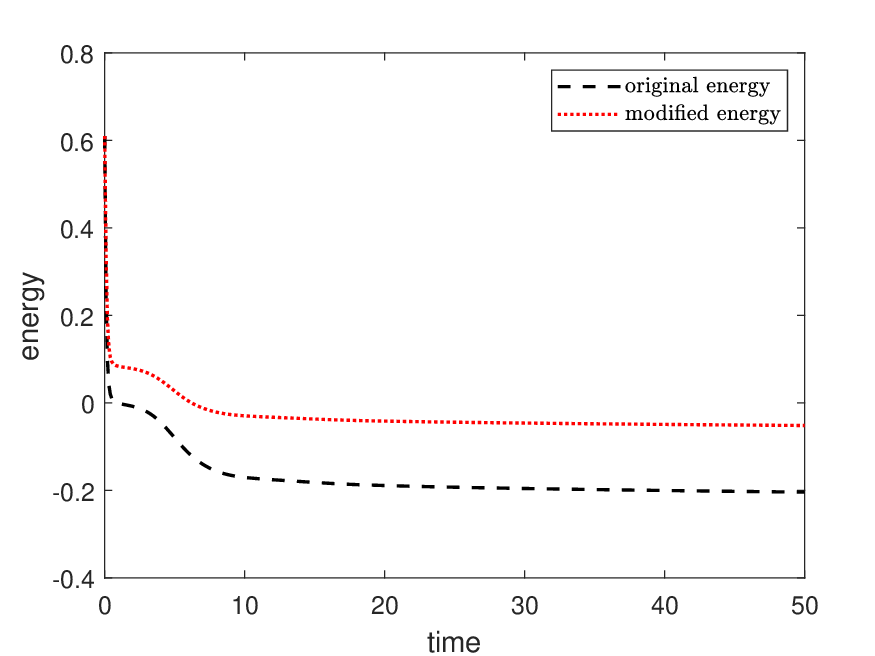}
		\label{figEx2_4a}
	}%
	\subfigure[$\tau = 0.01$]
	{
		\includegraphics[width=0.33\textwidth]{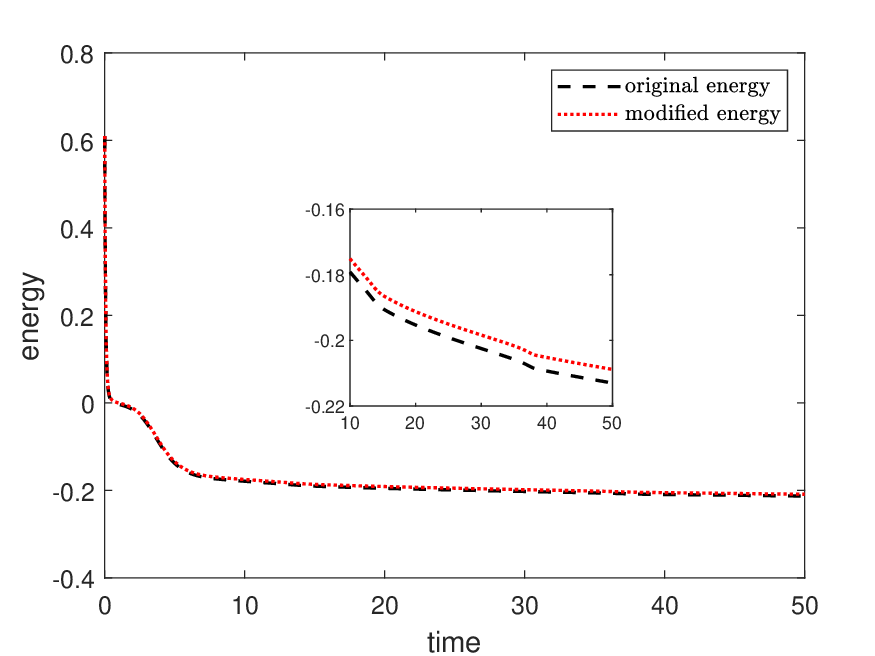}
		\label{figEx2_4b}
	}%
	\subfigure[$ \vert E_{h}{[\phi]} - \me_{h}{[\phi,R]} \vert $]
	{
		\includegraphics[width=0.33\textwidth]{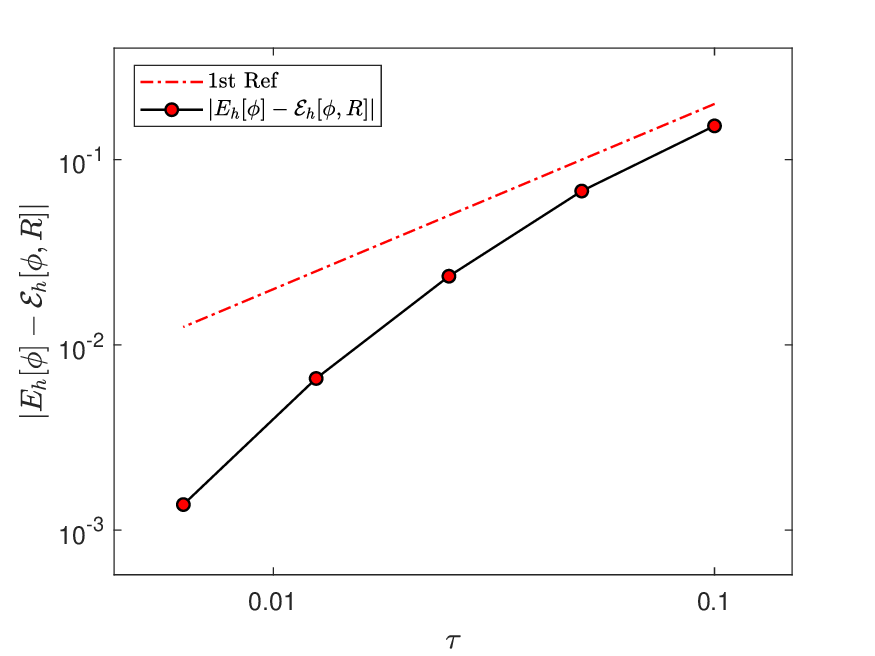}
		\label{figEx2_4c}
	}%
	\setlength{\abovecaptionskip}{0.0cm} 
	\setlength{\belowcaptionskip}{0.0cm}
	\caption{Evolutions of the original and modified energies wirh $\tau = 0.1$ (left) and $ \tau = 0.01 $ (middle), and the approximation rate of the modified energy to the original one (right) for the BDF2-sESAV-I scheme: the Flory--Huggins potential}
	\label{figEx2_4}
	\vspace{-10pt}
\end{figure}

\subsection{Long-term simulation with adaptive  time-stepping}\label{sec:ex3}
It is known that the coarsening dynamics process governed by the Allen--Cahn model usually requires a long time to reach the steady state  and its energy admits multiple time scales \cite{MOC_2023_Ju,SINUM_2020_Liao,SINUM_Ju_2022}. Thus, the adaptive time-stepping strategy is a heuristic and available method to improve the computational efficiency without sacrificing accuracy. We shall adopt the following time-stepping strategy \cite{JSC_Zhang_2022,SISC_2011_Qiao} based on the energy variation to adaptively select the next time steps
\begin{equation}\label{ex:adaptive}
	\begin{aligned}
		\tau_{n+1} = \min\bigg\{ \max \bigg\{ \tau_{\min} , \f{\tau_{\max}}{ \sqrt{ 1 + \alpha \vert \p_{\tau} E^{n} \vert^{2} } } \bigg\}, r_{\max} \tau_{n} \bigg\}
	\end{aligned}
\end{equation}
for the simulation of Example \ref{sec:ex2},
where $ \tau_{\min} $ and $ \tau_{\max} $ are predetermined minimum and maximum temporal stepsizes and $ \alpha $ is a tunable parameter. Here, we set $ r_{\max} = 2.4 $, $ \alpha = 10^{8} $, and choose $ ( \tau_{\min}, \tau_{\max} ) $ as $ ( 0.04, 0.4 ) $ for the double-well potential and $ (0.01, 0.1) $ for the Flory--Huggins potential, respectively. As shown above, the BDF2-sESAV-I scheme with stabilization \eqref{def:ST} is more reliable in preserving MBP compared to the other two stabilization methods, making it the preferred choice for the long-term coarsening dynamics simulation up to $ T = 2000 $. Though the discrete modified energy \eqref{def:dis_energy} is introduced as an approximation of the original one to facilitate the proof of energy dissipation law, we are more concerned about the original energy which reflects the real phase transition process, as done in \cite{SINUM_Ju_2022,JSC_Ju_2022}.

\begin{figure}[!htbp]
	\vspace{-10pt}
	\centering
	\subfigure[energy]
	{
		\includegraphics[width=0.33\textwidth]{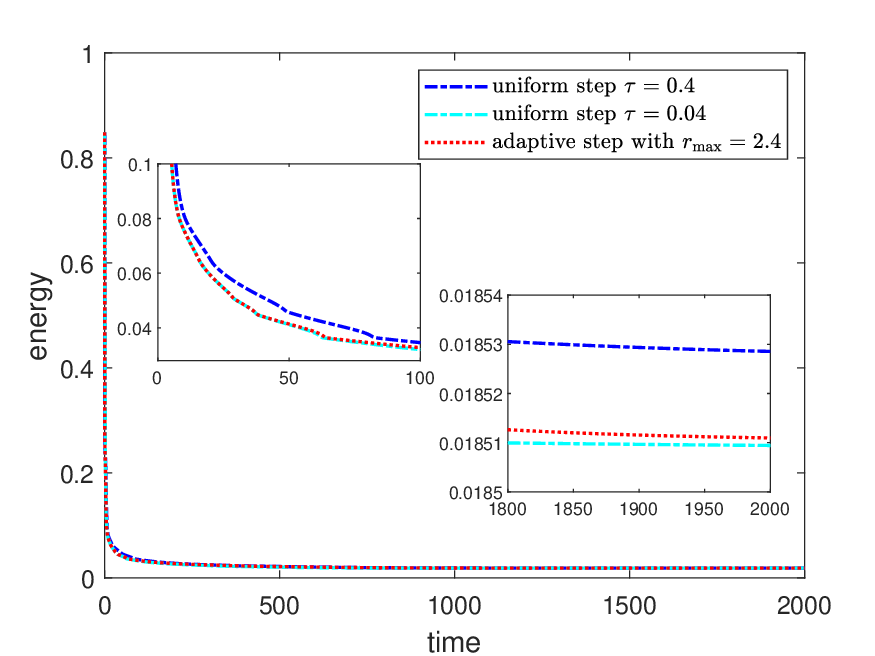}
		\label{figEx2_5a}
	}%
	\subfigure[maximum-norm of $\phi$]
	{
		\includegraphics[width=0.33\textwidth]{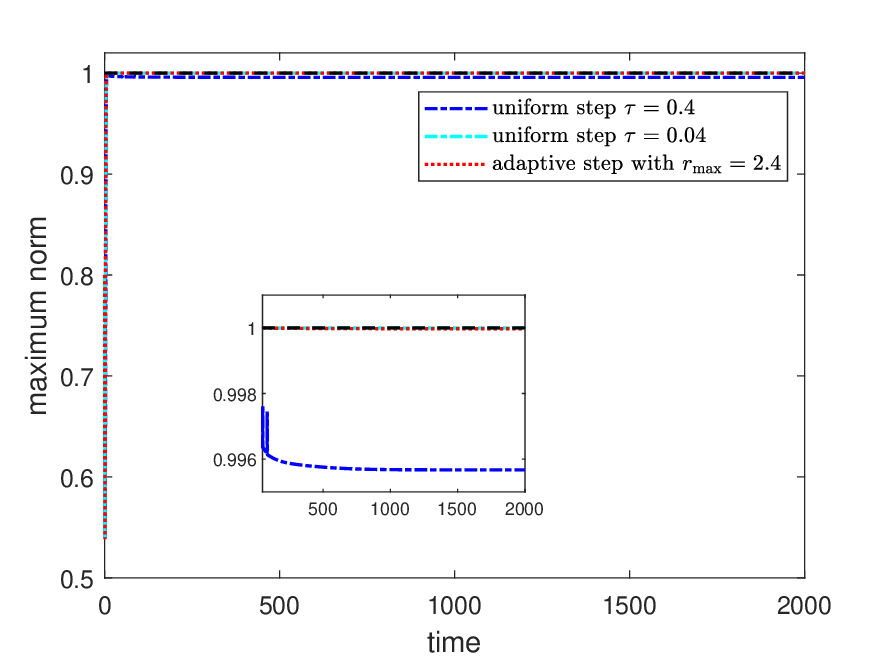}
		\label{figEx2_5b}
	}%
	\subfigure[temporal stepsize]
	{
		\includegraphics[width=0.33\textwidth]{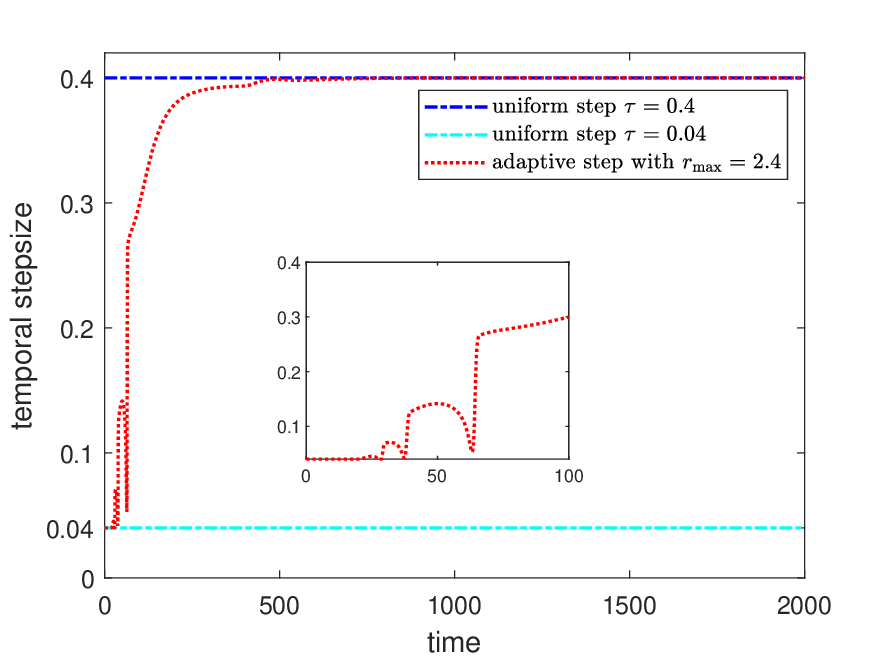}
		\label{figEx2_5c}
	}%
	\setlength{\abovecaptionskip}{0.0cm} 
	\setlength{\belowcaptionskip}{0.0cm}
	\caption{Evolutions of the original energy (left), the maximum-norm solution (middle), and the temporal stepsize (right) for the BDF2-sESAV-I scheme: the double-well potential}
	\label{figEx2_5}
		\vspace{-6pt}
\end{figure}
\begin{figure}[!htbp]
	\vspace{-10pt}
	\centering
	\subfigure[energy]
	{
		\includegraphics[width=0.33\textwidth]{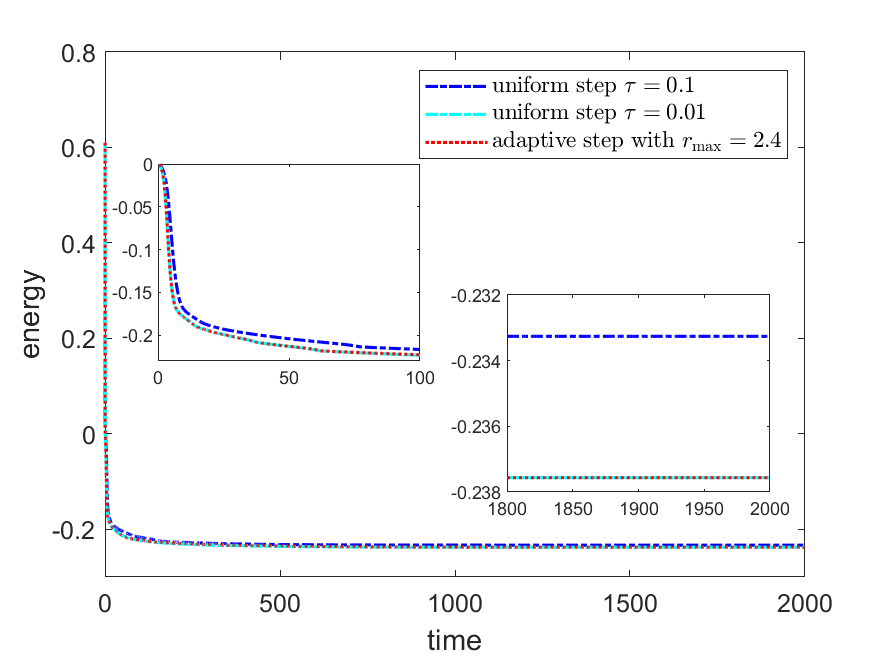}
		\label{figEx2_6a}
	}%
	\subfigure[maximum-norm of $\phi$]
	{
		\includegraphics[width=0.33\textwidth]{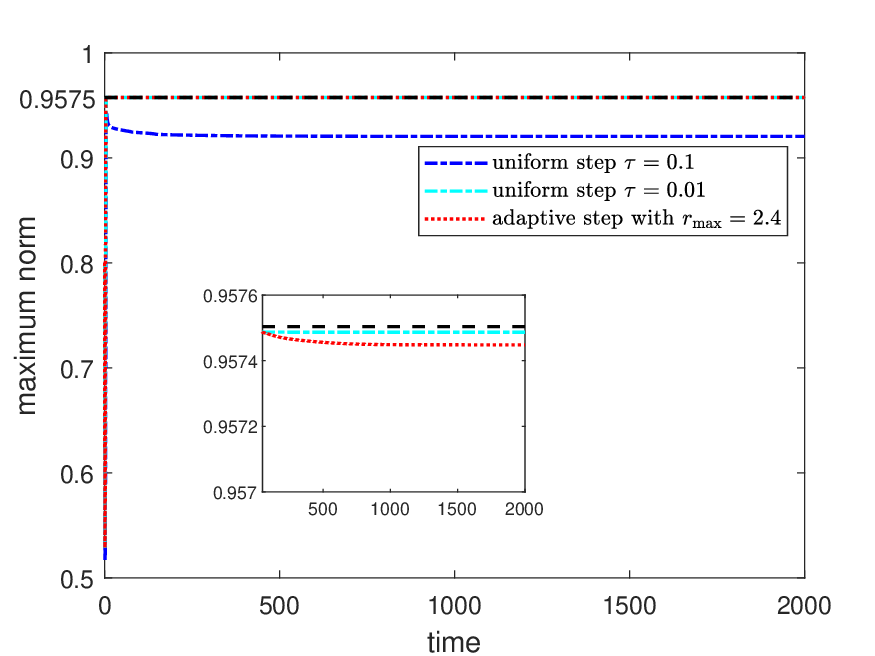}
		\label{figEx2_6b}
	}%
	\subfigure[temporal stepsize]
	{
		\includegraphics[width=0.33\textwidth]{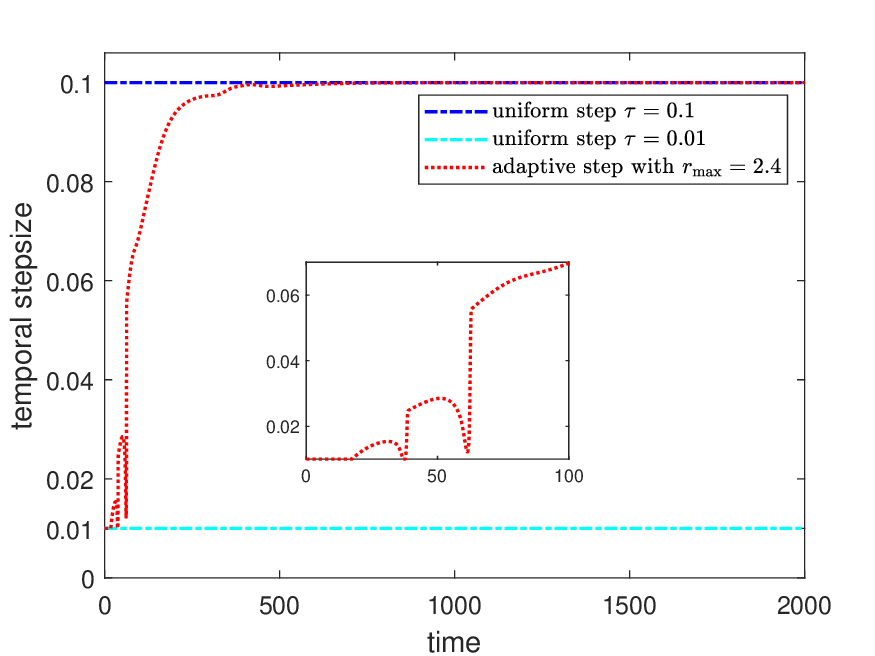}
		\label{figEx2_6c}
	}%
	\setlength{\abovecaptionskip}{0.0cm} 
	\setlength{\belowcaptionskip}{0.0cm}
	\caption{Evolutions of the original energy (left), the maximum-norm solution (middle), and the temporal stepsize (right)  for the BDF2-sESAV-I scheme: the Flory--Huggins potential}
	\label{figEx2_6}
	\vspace{-6pt}
\end{figure}
Numerical results on both uniform and adaptive non-uniform temporal grids are displayed in Fig. \ref{figEx2_5} for the double-well potential. It is clearly observed that the proposed BDF2-sESAV-I scheme preserves the energy dissipation law and MBP simultaneously, whether using uniform time stepsizes $\tau = 0.4$, $0.04$, or an adaptive time stepsize \eqref{ex:adaptive}. Moreover, Fig. \ref{figEx2_5c} shows that the temporal stepsizes are automatically selected to capture the changes of energy in the adaptive method, that is, during the time interval $ [ 0, 30 ] $, smaller temporal stepsizes are adopted since the energy decreases rapidly; after $ t = 30 $, the energy changes more and more slowly, and thus the temporal stepsize is increasing gradually; when $ t > 500 $, the temporal stepsize remains around $ 0.4 $. Furthermore, we find that although large temporal stepsize are often used in the adaptive method, the energy and maximum-norm of numerical solutions are in excellent agreement with those obtained using the small uniform stepsize $\tau = 0.04$. This indicates the effectiveness and efficiency of the adaptive BDF2-sESAV-I scheme, which can be further evidenced through Table \ref{tab1_Ex2}, where over $ 85 \% $ of the CPU time is saved. We also test the Flory--Huggins potential case, and the simulation results are displayed in Fig. \ref{figEx2_6} and Table \ref{tab1_Ex2}, from which similar phenomena and conclusions as before can be observed. Finally, the phase structures of these two potentials captured by the BDF2-sESAV-I scheme using the adaptive time stepsize \eqref{ex:adaptive} at some time instants are presented in Figs. \ref{figEx2_7} and \ref{figEx2_8}, respectively.

\begin{table}[!htbp]
	\vspace{-10pt}
	\caption{CPU times and the total number of time steps\label{Ex5_tab1} yielded by the BDF2-sESAV-I scheme}%
	\label{tab1_Ex2}
	{\footnotesize\begin{tabular*}{\columnwidth}{@{\extracolsep\fill}ccccc@{\extracolsep\fill}}
			\toprule
			& \multicolumn{2}{c}{double-well potential} & \multicolumn{2}{c}{Flory--Huggins potential}\\
			\midrule
			time-stepping strategy  & $N$ &  CPU times & $N$ & CPU times \\
			\midrule
			uniform step  $ \tau = \tau_{\max}$      &   5,000   &  2 m 24 s   &  20,000   &  10 m 56 s            \\ 
			adaptive step with $ r_{\max} = 2.4 $    &   6,018   &  3 m 15 s   &  24,171   &  13 m 51 s            \\ 
			uniform step  $ \tau = \tau_{\min} $     &   50,000  &  27 m 11 s  &  200,000  &  1 h 50 m 44 s          \\ 
			\bottomrule
	\end{tabular*}}
	\vspace{-8pt}
\end{table}
\begin{figure}[!htbp]
	\vspace{-8pt}
	\centering
	\subfigure[$t=5$]
	{
		\includegraphics[width=0.24\textwidth]{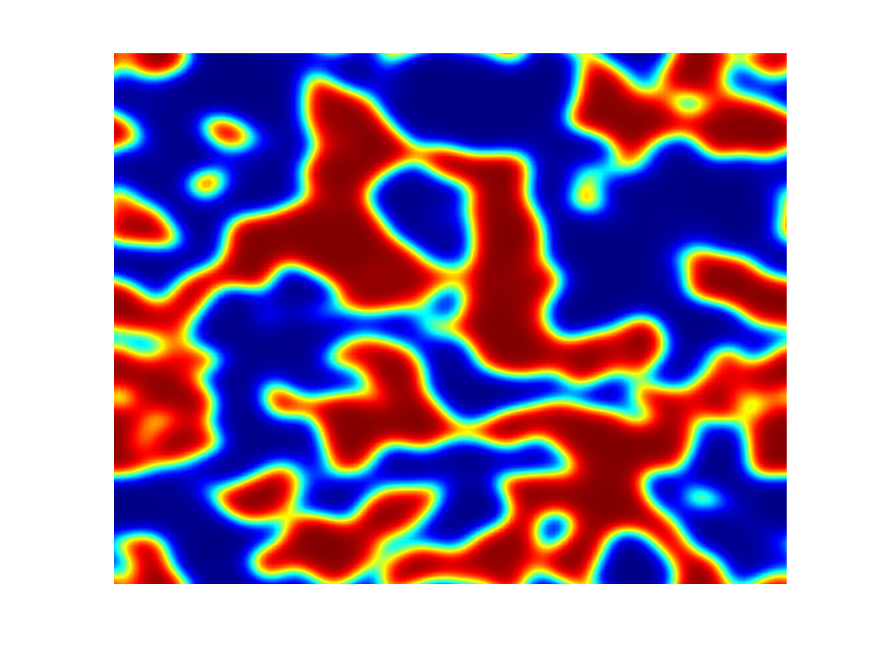}
		\label{figEx2_7a}
	}%
	\subfigure[$t=50$]
	{
		\includegraphics[width=0.24\textwidth]{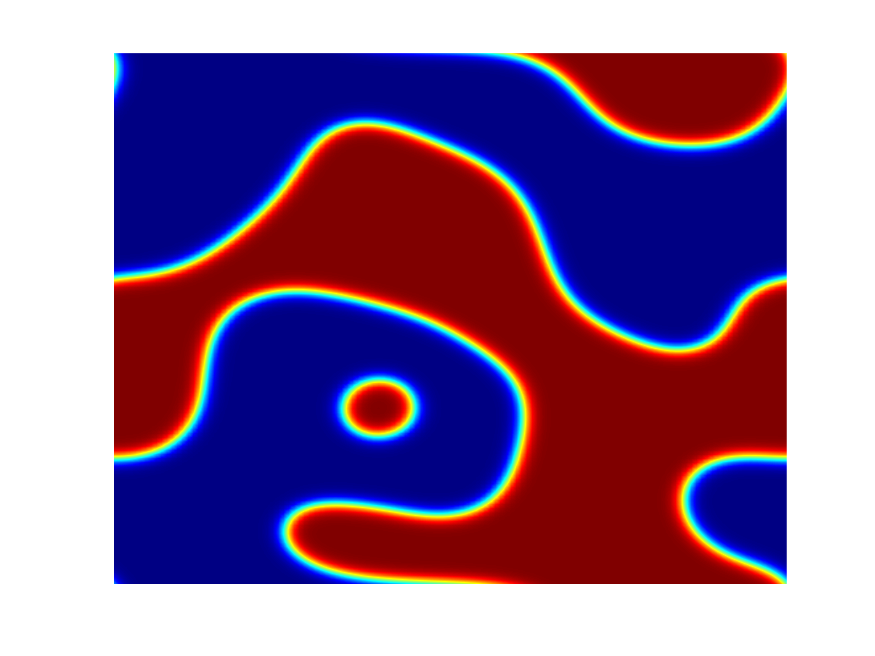}
		\label{figEx2_7b}
	}%
	\subfigure[$t=500$]
	{
		\includegraphics[width=0.24\textwidth]{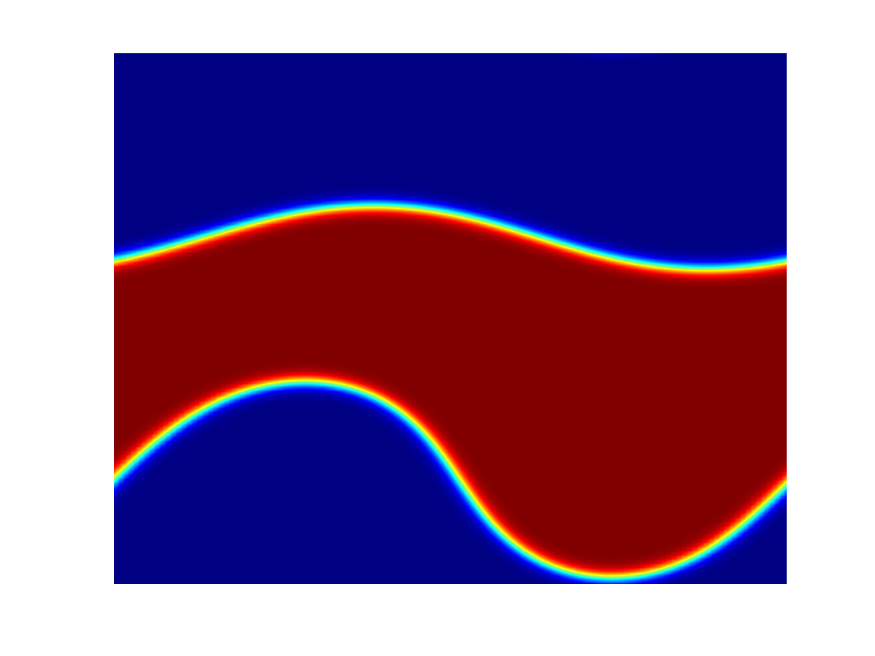}
		\label{figEx2_7c}
	}%
	\subfigure[$t=2000$]
	{
		\includegraphics[width=0.24\textwidth]{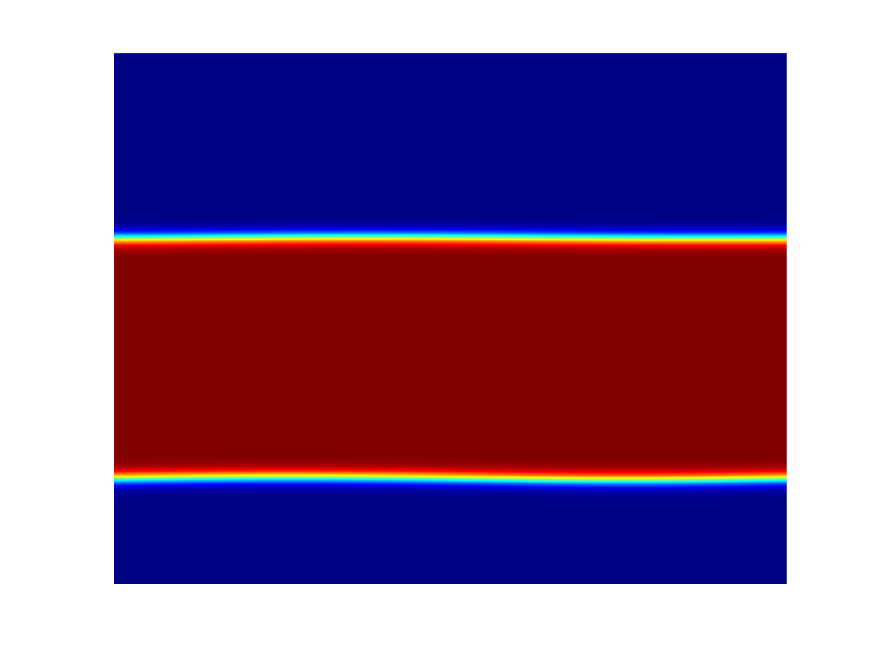}
		\label{figEx2_7d}
	}%
	\setlength{\abovecaptionskip}{0.0cm} 
	\setlength{\belowcaptionskip}{0.0cm}
	\caption{ Simulated phase structures by the BDF2-sESAV-I scheme with \eqref{ex:adaptive} for the coarsening dynamics of the double-well potential}
	\label{figEx2_7}
	\vspace{-8pt}
\end{figure}
\begin{figure}[!htbp]
	\vspace{-8pt}
	\centering
	\subfigure[$t=5$]
	{
		\includegraphics[width=0.24\textwidth]{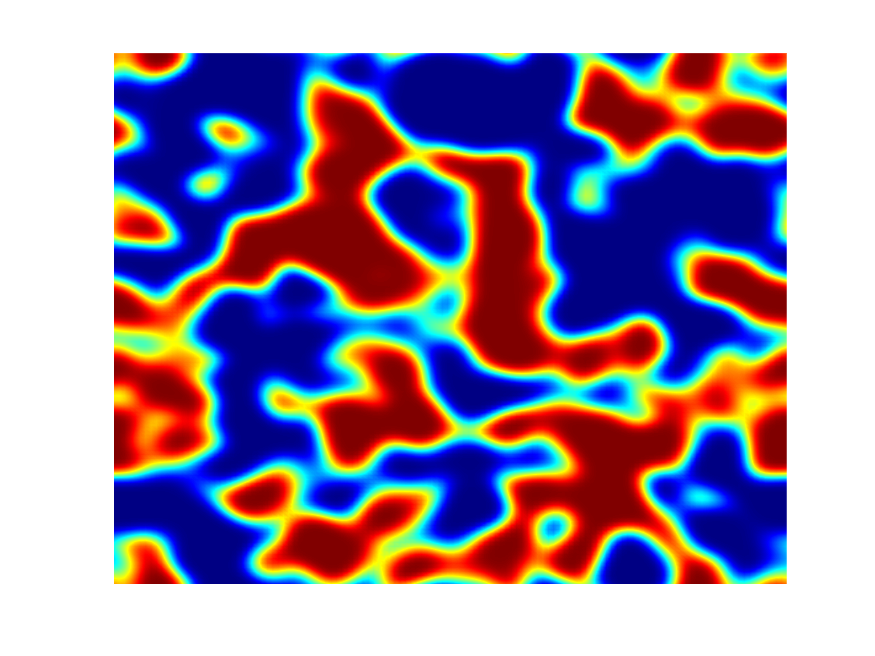}
		\label{figEx2_8a}
	}%
	\subfigure[$t=50$]
	{
		\includegraphics[width=0.24\textwidth]{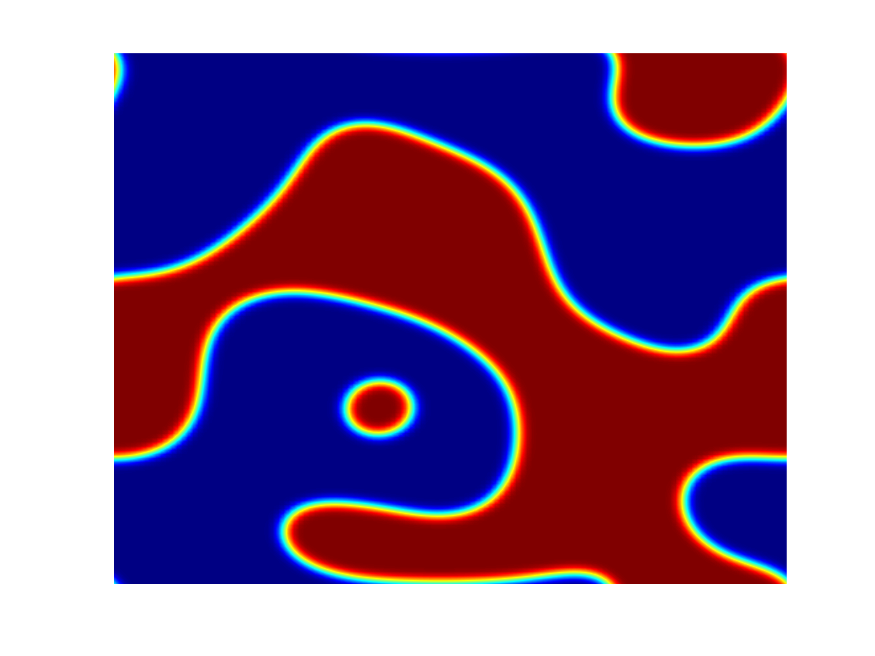}
		\label{figEx2_8b}
	}%
	\subfigure[$t=500$]
	{
		\includegraphics[width=0.24\textwidth]{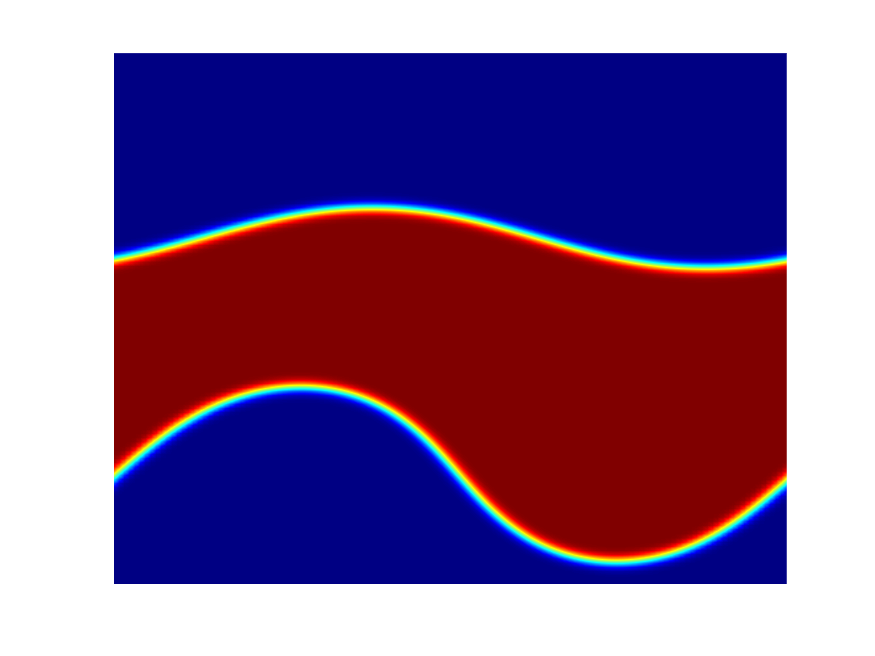}
		\label{figEx2_8c}
	}%
	\subfigure[$t=2000$]
	{
		\includegraphics[width=0.24\textwidth]{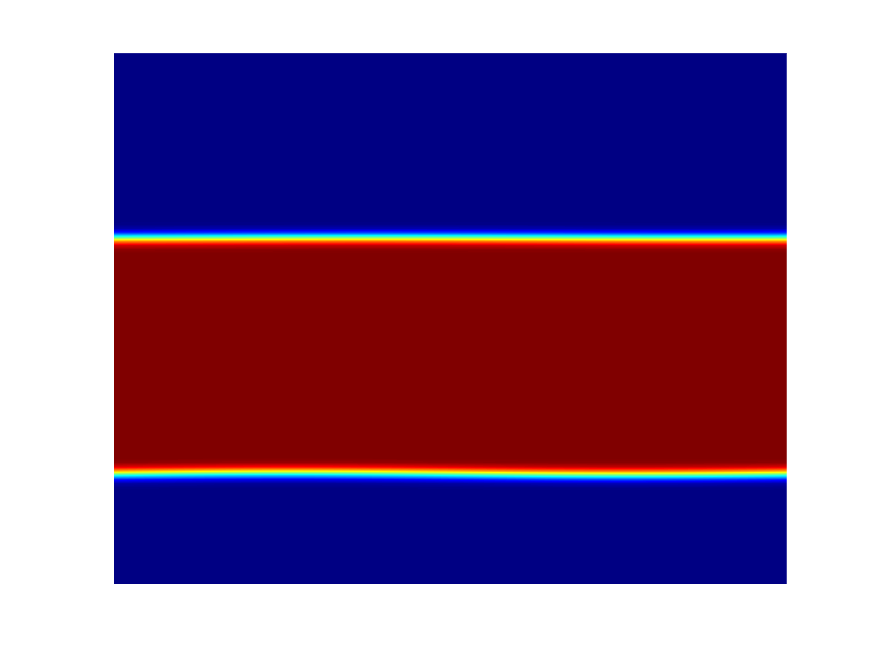}
		\label{figEx2_8d}
	}%
	\setlength{\abovecaptionskip}{0.0cm} 
	\setlength{\belowcaptionskip}{0.0cm}
	\caption{ Simulated phase structures by the BDF2-sESAV-I scheme with \eqref{ex:adaptive} for the coarsening dynamics of the Flory--Huggins potential}
	\label{figEx2_8}
	\vspace{-6pt}
\end{figure}

\section{Concluding remarks}
In this paper, we proposed several stabilization schemes that preserve both the discrete energy dissipation law and the MBP for the Allen--Cahn equation \eqref{Model:tAC}. These schemes are constructed using the variable-step BDF2 formula and the ESAV approach, enhanced with a novel stabilization technique. Specifically, we designed a novel essential auxiliary functional  satisfying assumptions (\textbf{A1})--(\textbf{A2}), which ensures us to construct variable-step second-order and unconditionally energy dissipative ESAV schemes. Furthermore, this auxiliary functional facilitates the construction of MBP-preserving stabilization schemes, of which the unbalanced BDF2-sESAV-I scheme is particularly recommended due to its exceptional ability to preserve the discrete MBP even for large trmporal stepsizes. Most importantly, by leveraging the kernel recombination technique,  we rigorously proved  the energy dissipation law and MBP-preservation for the variable-step BDF2-sESAV-I scheme under the conditions $ 0 < r_{k} < 4.864 - \delta $ and $ 0 < r_{k} < 1 + \sqrt{2} $, respectively, as detailed in Theorems \ref{thm:MBP_2} and \ref{thm:energy_2}. Moreover, Theorem \ref{thm:Err_Energy} presented the difference between the original energy and the modified energy of the proposed BDF2-sESAV-I scheme. Combined with error estimates for the auxiliary variable (cf. Lemma \ref{lem:estig}), we derived optimal error estimates in the discrete  $ H^{1}$- and $ L^{\infty} $-norms for the phase function, as presented in Theorems \ref{thm:ErrR} and \ref{thm:ErrPhi}. These theoretical claims were validated through extensive numerical experiments, confirming the accuracy and effectiveness of the proposed scheme.

%
%

\section*{Declarations}
\begin{itemize}
	\item \textbf{Funding}~   This work was supported in part by the National Natural Science Foundation of China (Nos. 12131014, 12301531, 12371412), by the Shandong Provincial Natural Science Foundation (Nos. ZR2024MA023, 2023HWYQ-064), by the Fundamental Research Funds for the Central Universities (Nos. 202264006, 202461103), and by the OUC Scientific Research Program for Young Talented Professionals.
	\item  \textbf{Data Availability}~  Enquiries about data availability should be directed to the authors.
	\item \textbf{Conflict of interest}~  The authors declare that they have no competing interests.
\end{itemize}

\appendix

\section{Proof of Lemma \ref{lem:estihat}}\label{App:B}
\setcounter{equation}{0}
\renewcommand\theequation{A.\arabic{equation}}
It is easy to check that the exact solution $ \phi( t ) $ of \eqref{Model:tAC} satisfies
\begin{equation}\label{Errh_1}
		\md_{1} \phi( t_n )  = \varepsilon^2 \Delta_h \phi( t_n ) + f( \phi( t_{n-1} ) ) - \kappa ( \phi( t_n ) - \phi( t_{n-1} ) ) + T_{1}^n,
\end{equation}
where the local truncation error $ T_{1}^n $ can be estimated by
\begin{equation}\label{Trun_11}
	\| T_{1}^{n} \|_{\infty} \leq C \left( \tau_{n} \| \phi \|_{ W^{2,\infty}(0,T;L^{\infty}(\Omega)) } + h^2 \| \phi \|_{ L^{\infty}(0,T;W^{4,\infty}(\Omega)) } \right).
\end{equation}
Then the difference between \eqref{sch:2_3} and \eqref{Errh_1} leads to
\begin{equation}\label{Errh_2}
		\hat{e}_{\phi}^{n} - e_{\phi}^{n-1}  = \varepsilon^2 \tau_{n} \Delta_h \hat{e}_{\phi}^{n} + \tau_{n} ( f( \phi( t_{n-1} ) ) - f( \phi^{n-1} ) ) - \kappa \tau_{n} ( \hat{e}_{\phi}^{n} - e_{\phi}^{n-1} ) + \tau_{n} T_{1}^n,  
\end{equation}
for $ n \geq 1$.

On the first hand, taking the discrete inner product of \eqref{Errh_2} with $ \hat{e}_{\phi}^{n} $ and rearranging the terms, we get
\begin{equation*} 
		( 1 + \kappa \tau_{n} ) \| \hat{e}_{\phi}^{n} \|^{2}  \leq 	( 1 + \kappa \tau_{n} ) \big\langle  e_{\phi}^{n-1}, \hat{e}_{\phi}^{n} \big\rangle +  \tau_{n} \big\langle f( \phi( t_{n-1} ) ) - f( \phi^{n-1} ), \hat{e}_{\phi}^{n} \big\rangle  + \tau_{n} \big\langle T_{1}^n, \hat{e}_{\phi}^{n} \big\rangle,
\end{equation*}
which, by eliminating $ \| \hat{e}_{\phi}^{n} \| $ and using \eqref{Bound_expf} and \eqref{Trun_11}, leads to
\begin{equation*}
	\begin{aligned}
		( 1 + \kappa \tau_{n} ) \| \hat{e}_{\phi}^{n} \| & \leq 	( 1 + \kappa \tau_{n} ) \| e_{\phi}^{n-1} \| +  \tau_{n} \| f( \phi( t_{n-1} ) ) - f( \phi^{n-1} ) \| + \tau_{n} \| T_{1}^n \| \\
		& \leq ( 1 + \kappa \tau_{n} + \overline{K} \tau_{n} ) \| e_{\phi}^{n-1} \| + C \tau_{n} ( \tau_{n} + h^2 ),
	\end{aligned}
\end{equation*}
Thus, the first conclusion of Lemma \ref{lem:estihat} is proved.

On the other hand, we equivalently rewrite \eqref{Errh_2} as
\begin{equation*}
		\left[ ( 1 + \kappa \tau_{n} ) I - \tau_{n} \varepsilon^2 \Delta_h \right]\hat{e}_{\phi}^{n} = ( 1 + \kappa \tau_{n} ) e_{\phi}^{n-1} + \tau_{n} ( f( \phi( t_{n-1} ) ) - f( \phi^{n-1} ) ) + \tau_{n} T_{1}^n.
\end{equation*}
Then, a direct application of Lemma \ref{lem:MBP_left} gives
\begin{equation*}
		( 1 + \kappa \tau_{n} ) \| \hat{e}_{\phi}^{n} \|_{\infty} \leq ( 1 + \kappa \tau_{n} + \overline{K} \tau_{n} )  \| e_{\phi}^{n-1} \|_{\infty}  + C \tau_{n} ( \tau_{n} + h^2 ),
\end{equation*}
which proves the second conclusion.

\section{Proof of Theorem \ref{thm:ErrR}}\label{App:C}
\setcounter{equation}{0}
\renewcommand\theequation{B.\arabic{equation}}
First, taking the discrete inner product of \eqref{ErrR_3} with $ \nabla_{\tau} e_\phi^{n} $ and rearranging the term lead to
\begin{equation}\label{ErrR_5}
	\begin{aligned}
		& \quad \big\langle G[ \nabla_{\tau} e_\phi^{n} ], 1 \big\rangle + \frac{\varepsilon^2 }{2} \| \nabla_{h} e_\phi^{n} \|^2 + \frac{ \delta \| \nabla_{\tau} e_\phi^{n} \|^2 }{32 \tau_{n} } 
		- \big\langle G[ \nabla_{\tau} e_\phi^{n-1} ], 1 \big\rangle 
		- \frac{\varepsilon^2 }{2} \| \nabla_{h} e_\phi^{n-1} \|^2 
		\\
		& \leq \big\langle f( \phi( t_{n} ) ) - V( g_h(\hat{\phi}^{n}, R^{n-1}) )   f(\hat{\phi}^{n}), \nabla_{\tau} e_\phi^{n} \big\rangle \\
		& \quad + \kappa \big\langle \phi^{n} - V ( g_h(\hat{\phi}^{n}, R^{n-1}) ) \hat{\phi}^{n}, \nabla_{\tau} e_\phi^{n} \big\rangle + \big\langle T_2^n, \nabla_{\tau} e_\phi^{n} \big\rangle\\
		& \leq \big\langle f( \phi( t_{n} ) ) - V( g_h(\hat{\phi}^{n}, R^{n-1}) )   f(\hat{\phi}^{n}), \nabla_{\tau} e_\phi^{n} \big\rangle 
		- \frac{\kappa }{2} \| \nabla_{\tau} e_\phi^{n} \|^2 - \frac{\kappa }{2} \big\langle e_{\phi}^{n}, \nabla_{\tau} e_\phi^{n} \big\rangle \\
		& \quad  + \kappa \big\langle \hat{e}_{\phi}^{n} - \frac{1}{2} e_{\phi}^{n-1}, \nabla_{\tau} e_\phi^{n} \big\rangle +  \kappa \big\langle ( 1 - V ( g_h(\hat{\phi}^{n}, R^{n-1}) ) ) \hat{\phi}^{n}, \nabla_{\tau} e_\phi^{n} \big\rangle  + \big\langle T_2^n, \nabla_{\tau} e_\phi^{n} \big\rangle  \\
		& =: \sum^{6}_{i=1} I_{i},
	\end{aligned}
\end{equation}
where we have used Lemma \ref{lem:BDF2_P1} and the identity $ 2 a(a-b) \geq a^2 - b^2 $ and the fact
\begin{equation}\label{ErrR_phi_vphi}
		\phi^{n} - V ( g_h(\hat{\phi}^{n}, R^{n-1}) ) \hat{\phi}^{n} = -\frac{1}{2}\nabla_{\tau} e_{\phi}^{n} - \frac{1}{2} e_{\phi}^{n} + \hat{e}_{\phi}^{n} - \frac{1}{2} e_{\phi}^{n-1} 
		+ ( 1 - V ( g_h(\hat{\phi}^{n}, R^{n-1}) ) ) \hat{\phi}^{n}.
\end{equation}

For the first term $ I_{1} $, we use Young's inequality to get
\begin{equation*}
	\begin{aligned}
		I_1 & \leq \frac{ 32 \tau_{n} }{ \delta } \| f( \phi( t_{n} ) ) - V( g_h(\hat{\phi}^{n}, R^{n-1}) )   f(\hat{\phi}^{n}) \|^2 + \frac{ \delta \| \nabla_{\tau} e_\phi^{n} \|^2 }{128 \tau_{n} } \\
		& \leq \frac{ 64 \tau_{n} }{ \delta } \left( \| f( \phi( t_{n} ) ) -    f(\hat{\phi}^{n}) \|^2 + \| ( V( g_h(\hat{\phi}^{n}, R^{n-1}) ) - 1 )   f(\hat{\phi}^{n}) \|^2 \right) + \frac{ \delta \| \nabla_{\tau} e_\phi^{n} \|^2 }{128 \tau_{n} },
	\end{aligned}
\end{equation*}
which together with \eqref{Bound_expf} and Lemma \ref{lem:estig}  yields
\begin{equation*}\label{ErrR_6}
		I_1 \leq C \tau_{n} \left( \| \hat{e}_\phi^{n} \|^2 + \vert e_R^{n-1} \vert^2 + \tau_{n}^2 + h^4 \right) + \frac{ \delta \| \nabla_{\tau} e_\phi^{n} \|^2 }{128 \tau_{n} }.
\end{equation*}
Moreover, by Theorem \ref{thm:MBP_2} and Lemma \ref{lem:estig}, the last three terms can be similarly estimated by
\begin{equation*}\label{ErrR_7}
		\sum^{6}_{i=4} I_{i} 
		\leq  C \tau_{n} \left( \| \hat{e}_{\phi}^{n}  \|^2 + \| e_{\phi}^{n-1} \|^2 + \vert e_R^{n-1} \vert^2 + \| T_2^n \|^2 + \tau_{n}^2 + h^4 \right) + \frac{ \delta \| \nabla_{\tau} e_\phi^{n} \|^2 }{128 \tau_{n} }.
\end{equation*}

Thus, substituting the estimates $I_1$ and $I_4-I_6$ into \eqref{ErrR_5} and considering the identity $ 2 a(a-b) \geq a^2 - b^2 $ give us
\begin{equation}\label{ErrR_8}
	\begin{aligned}
		& \big\langle G[ \nabla_{\tau} e_\phi^{n} ], 1 \big\rangle + \frac{\kappa}{4} \| e_{\phi}^{n} \|^2 + \frac{\varepsilon^2 }{2} \| \nabla_{h} e_\phi^{n} \|^2 + \frac{ \delta \| \nabla_{\tau} e_\phi^{n} \|^2 }{64 \tau_{n} } + \frac{\kappa}{2} \| \nabla_{\tau} e_\phi^{n} \|^2  
		\\
		& \quad \leq \big\langle G[ \nabla_{\tau} e_\phi^{n-1} ], 1 \big\rangle  + \frac{\kappa}{4} \| e_{\phi}^{n-1} \|^2 
		+ \frac{\varepsilon^2 }{2} \| \nabla_{h} e_\phi^{n-1} \|^2 
		\\
		& \qquad + C \tau_{n} \left( \| \hat{e}_{\phi}^{n}  \|^2 + \| e_{\phi}^{n-1} \|^2 + \vert e_R^{n-1} \vert^2  + \| T_2^n \|^2 + \tau_{n}^2 + h^4 \right).
	\end{aligned}
\end{equation}

Second, multiplying \eqref{ErrR_4} by $ 2 \tau_{n} e_{R}^{n} $ yields
\begin{equation}\label{ErrR_9}
	\begin{aligned}
		\vert e_R^{n} \vert^2 - \vert e_R^{n-1} \vert^2 
		&  \le - 2 e_R^{n} V ( g_h(\hat{\phi}^{n}, R^{n-1}) ) \big\langle f(\hat{\phi}^{n}), \nabla_{\tau} e_{\phi}^n \big\rangle \\
		& \quad  + 2 e_R^{n} \big\langle V ( g_h(\hat{\phi}^{n}, R^{n-1}) ) f(\hat{\phi}^{n}) - f(\phi(t_n)), \nabla_{\tau} \phi( t_n ) \big\rangle \\
		& \quad + 2 \kappa e_R^{n} \big\langle \phi^{n} - V ( g_h(\hat{\phi}^{n}, R^{n-1}) ) \hat{\phi}^{n},  \nabla_{\tau} e_{\phi}^n \big\rangle \\
		& \quad - 2 \kappa e_R^{n} \big\langle \phi^{n} - V ( g_h(\hat{\phi}^{n}, R^{n-1}) ) \hat{\phi}^{n},  \nabla_{\tau} \phi( t_n )  \big\rangle 
		+ 2 \tau_{n} e_R^{n} T_{3}^n
		 := \sum^{5}_{i=1} J_{i}.
	\end{aligned}
\end{equation}

By the boundedness of $ \phi^n $, $ \hat{\phi}^n $, \eqref{Bound_expf} and the assumption (\textbf{A2}), the first and third terms can be estimated similarly to the analysis of $I_{4}$--$I_{6}$ as follows
\begin{equation*}
		J_1 + J_3  \leq 2 \overline{K} \vert \Omega \vert \vert e_R^{n} \vert \| \nabla_{\tau} e_{\phi}^n \| + 4 \kappa \beta \vert \Omega \vert \vert e_R^{n} \vert \| \nabla_{\tau} e_{\phi}^n \| \leq \hat{C} \tau_{n} \vert e_{R}^{n} \vert^2 + \frac{ \delta \| \nabla_{\tau} e_\phi^{n} \|^2 }{64 \tau_{n} }.
\end{equation*}

From the estimate process of $I_{1}$, we see
\begin{equation*}
		\| f( \phi( t_{n} ) ) - V( g_h(\hat{\phi}^{n}, R^{n-1}) )   f(\hat{\phi}^{n}) \|^2 \leq C \left( \| \hat{e}_\phi^{n} \|^2 + \vert e_R^{n-1} \vert^2 + \tau_{n}^2 + h^4 \right),
\end{equation*}
and then, by Cauchy-Schwarz inequality, mean value theorem and the above result, we have 
\begin{equation*}
	\begin{aligned}
		J_2 + J_5 & = 2 e_R^{n} \big\langle V ( g_h(\hat{\phi}^{n}, R^{n-1}) ) f(\hat{\phi}^{n}) - f(\phi(t_n)), \nabla_{\tau} \phi( t_n ) \big\rangle + 2 \tau_{n} e_R^{n} T_{3}^n \\
		&  \leq \frac{1}{2} \tau_n \vert e_R^{n} \vert^2 + C \tau_{n} \left( \| \hat{e}_\phi^{n} \|^2 + \vert e_R^{n-1} \vert^2 + \vert T_3^n \vert^2 + \tau_{n}^2 + h^4 \right).
	\end{aligned}
\end{equation*}

Finally, using the substitution formula \eqref{ErrR_phi_vphi}, $ J_{4} $ can be estimated by
\begin{equation*}
	\begin{aligned}
		J_4 & = - 2 \kappa e_R^{n} \big\langle -\nabla_{\tau} e_{\phi}^{n} + \hat{e}_{\phi}^{n} - e_{\phi}^{n-1} + ( 1 - V ( g_h(\hat{\phi}^{n}, R^{n-1}) ) ) \hat{\phi}^{n}, \nabla_{\tau} \phi( t_n )  \big\rangle \\
		& \leq 2 \tau_{n} \kappa \vert e_R^{n} \vert \left( \| \nabla_{\tau} e_{\phi}^{n} \| + \| \hat{e}_{\phi}^{n} \| + \| e_{\phi}^{n-1} \| + \big\vert 1 - V ( g_h(\hat{\phi}^{n}, R^{n-1}) ) \big\vert \beta \right) \| \phi_{t} \|_{ L^{\infty}( 0,T;L^{2}(\Omega) ) },
	\end{aligned}
\end{equation*}
which together with Lemma \ref{lem:estig}  and Young's inequality implies
\begin{equation*}
	\begin{aligned}
		J_4 & \leq \Big( 2 \kappa \tau_{n}^2 \| \phi_{t} \|_{ L^{\infty}( 0,T;L^{2}(\Omega) ) }^{2} + \frac{1}{4} \tau_{n} \Big) \vert e_R^{n} \vert^2  
		+ \frac{\kappa}{2} \| \nabla_{\tau} e_{\phi}^{n} \|^2  \\
		&\qquad + C \tau_{n } \Big( \| \hat{e}_{\phi}^{n} \|^2 + \| e_{\phi}^{n-1} \|^2 + \vert e_{R}^{n-1} \vert^2 + \tau_{n}^{2} + h^4 \Big) \\
		& \leq \frac{1}{2} \tau_{n} \vert e_R^{n} \vert^2  + \frac{\kappa}{2} \| \nabla_{\tau} e_{\phi}^{n} \|^2 
		 + C \tau_{n } \Big( \| \hat{e}_{\phi}^{n} \|^2 + \| e_{\phi}^{n-1} \|^2 + \vert e_{R}^{n-1} \vert^2 + \tau_{n}^{2} + h^4 \Big),
	\end{aligned}
\end{equation*}
under the temporal stepsize condition $ \tau_{n} \leq  1 / (8\kappa \| \phi_{t} \|_{ L^{\infty}(0,T;L^{2}(\Omega)) }^2 + 1)$.

Combining \eqref{ErrR_9} with the estimates of $J_{1}$--$J_{5}$, we obtain
\begin{equation}\label{ErrR_13}
	\begin{aligned}
		\vert e_R^{n} \vert^2 - \vert e_R^{n-1} \vert^2 & \leq  ( \hat{C} + 1 ) \tau_{n} \vert e_R^{n} \vert^2 + \frac{\kappa}{2} \| \nabla_{\tau} e_{\phi}^{n} \| + \frac{ \delta \| \nabla_{\tau} e_\phi^{n} \|^2 }{64 \tau_{n} } \\
		& \quad + C \tau_{n} \big( \| \hat{e}_\phi^{n} \|^2 + \| e_{\phi}^{n-1} \|^2 + \vert e_R^{n-1} \vert^2 + \vert T_3^n \vert^2 + \tau_{n}^2 + h^4 \big),
	\end{aligned}
\end{equation}
where the constant $ \hat{C} := ( 64 \overline{K}^2 + 256  \kappa^2 \beta^2)   \vert \Omega \vert \delta^{-1}$.

Let $ W^{n} := \big\langle G[ \nabla_{\tau} e_\phi^{n} ], 1 \big\rangle + \frac{\kappa}{4} \| e_\phi^{n} \|^2 + \frac{\varepsilon^2 }{2} \| \nabla_{h} e_\phi^{n} \|^2  + \vert e_R^{n} \vert^2 $ for $ n \geq 1 $. Then, adding \eqref{ErrR_8} and \eqref{ErrR_13} together, and utilizing  Lemma \ref{lem:estihat} yield
\begin{equation*}
		W^{n} - W^{n-1} 
		\leq ( \hat{C} + 1 ) \tau_{n} W^{n} +  C \tau_{n} \left( W^{n-1}  + \| T_2^n \|^2 + \vert T_3^n \vert^2 + \tau_{n}^2 + h^4 \right).
\end{equation*}
Thus, we sum up the above inequality from $2$ to $n$ to get
\begin{equation}\label{ErrR_15}
	W^{n} - W^{1} \leq ( \hat{C} + 1 ) \tau_{n} W^{n} + C \sum^{n}_{k=2} \tau_{k} \left( W^{k-1}  + \| T_2^{k} \|^2 
	+ \vert T_3^{k} \vert^2 + \tau_{k}^2 + h^4 \right),~  n \geq 2.
\end{equation}

Note that \eqref{ErrR_3}--\eqref{ErrR_4} also hold for $n=1$ with 
$\md_{2} e_\phi^{1} = \md_{1} e_\phi^{1}$.
Thus, following the similar argument as case $ n \geq 2 $, we can easily deduce the following estimate
\begin{equation}\label{ErrR_15_2}
	\begin{aligned}
		& \big\langle \md_{1} e_\phi^{1}, \nabla_{\tau} e_\phi^{1}  \big\rangle  + \frac{\kappa}{4} \| e_\phi^{1} \|^2 + \frac{\varepsilon^2 }{2} \| \nabla_{h} e_\phi^{1} \|^2 + \vert e_R^{1} \vert^2 \\
		& \quad \leq \frac{ \delta }{ 32 \tau_{1} } \| \nabla_{\tau} e_\phi^{1} \|^2 +  ( \hat{C} + 1 ) \tau_{1} \vert e_R^{1} \vert^2 +  C \tau_{1} \left( \| T_2^1 \|^2 + \vert T_3^1 \vert^2 + \tau_{1}^2 + h^4 \right),
	\end{aligned}
\end{equation}
under condition $ \tau_{1} \leq  1 / (8\kappa \| \phi_{t} \|_{ L^{\infty}(0,T;L^{2}(\Omega)) }^2 + 1)$. Furthermore, using \eqref{ErrT_21} with $n=1$ gives
\begin{equation}\label{ErrR_16}
   W^{1} \leq  C ( \tau_{1}^3 + h^4 ) \leq C ( \tau^4 + h^4 ),~~ \text{if}~  \tau_{1} \leq \min\big\{ \frac{1}{ 2 ( \hat{C} + 1 ) }, \tau^{4/3} \big\}.
\end{equation}
Now, inserting \eqref{ErrR_16} and the truncation errors \eqref{ErrT_21} into \eqref{ErrR_15}, and together with the discrete Gr\"onwall’s lemma, we can derive the desired estimate \eqref{thmErrR:1} under conditions $ \tau_{1} \leq \tau^{4/3} $ and $ \tau_{n} \leq  \hat{ \tau }:=\min \big\{ \frac{ 1 }{ 8\kappa \| \phi_{t} \|_{ L^{\infty}(0,T;L^{2}(\Omega)) }^2 + 1 }, \frac{1}{ 2 ( \hat{C} + 1 ) } \big\} $. 

Moreover, due to \eqref{thmErrR:1} and Lemma \ref{lem:estihat}, we have $ \| \hat{e}_{\phi}^n \| \leq C ( \tau + h^2 ) $, and thus
Lemma \ref{lem:estig}  reduces to
\begin{equation}\label{ErrR_17}
	\big\vert V( g_h( \hat{\phi}^n, R^{n-1} ) ) - 1 \big\vert \leq C \big( \| \hat{e}_{\phi}^n \| + \vert e_{R}^{n-1} \vert + \tau + h^2 \big)^2 \leq C ( \tau^2 + h^4 ).
\end{equation}
Furthermore, the result \eqref{ErrR_17} can further improve the above estimates of $I_{1}$ and $ I_{4} $--$I_{6}$ to
\begin{equation*} 
		I_1  \leq C \tau_{n} \Big( \| \hat{e}_\phi^{n} \|^2 + \big\vert V( g_h( \hat{\phi}^n, R^{n-1} ) ) - 1 \big\vert^2 \Big) + \frac{ \delta \| \nabla_{\tau} e_\phi^{n} \|^2 }{64 \tau_{n} } \leq C \tau_{n} ( \| \hat{e}_\phi^{n} \|^2 +  \tau^4 + h^4 ) + \frac{ \delta \| \nabla_{\tau} e_\phi^{n} \|^2 }{64 \tau_{n} },
\end{equation*}
\begin{equation*}
		I_4 + I_5 + I_{6} \leq  C \tau_{n} \left( \| \hat{e}_\phi^{n} \|^2 + \| T_{2}^{n} \|^2 + \tau^4 + h^4 \right) + \frac{ \delta \| \nabla_{\tau} e_\phi^{n} \|^2 }{64 \tau_{n} }.
\end{equation*}
Therefore, substituting the above improved estimates and \eqref{ErrT_21} into \eqref{ErrR_5} and using Lemma \ref{lem:estihat}, we  obtain the following improved version estimate of \eqref{ErrR_8}:
\begin{equation*}
	\begin{aligned}
		& \big\langle G[ \nabla_{\tau} e_\phi^{n} ], 1 \big\rangle + 
		\frac{\kappa}{4} \| e_\phi^{n} \|^2 + \frac{\varepsilon^2 }{2} \| \nabla_{h} e_\phi^{n} \|^2 
		\\
		& \quad \leq  \big\langle G[ \nabla_{\tau} e_\phi^{n-1} ], 1 \big\rangle + \frac{\kappa}{4} \| e_\phi^{n-1} \|^2
		+\frac{\varepsilon^2 }{2} \| \nabla_{h} e_\phi^{n-1} \|^2+ C \tau_{n} \big( \| e_{\phi}^{n-1} \|^2    + \tau^4 + h^4 \big),
	\end{aligned}
\end{equation*}
which implies the desired estimate \eqref{thmErrR:2} by applying the discrete Gr\"onwall’s lemma with initial estimate \eqref{ErrR_16}. Thus, we complete the proof.

\section{Proof of Theorem \ref{thm:Err_Energy}}\label{App:A}
\setcounter{equation}{0}
\renewcommand\theequation{C.\arabic{equation}}
The claimed inequality \eqref{thmErrE:0} is clearly valid for $n=0$. It remains to consider the case $ n \geq 1 $. It follows from \eqref{def:dis_energy} that
\begin{equation}\label{thmErrE:2}
	\vert E_h[\phi^{n}] - \me_h[\phi^{n}, R^{n}] \vert \leq \langle G[ \nabla_{\tau} \phi^{n} ], 1 \rangle + \vert  E_{1h}[ \phi^{n} ] - R^n \vert,
\end{equation}
in which  Theorem \ref{thm:ErrR} implies that
\begin{equation*} 
		\langle G[ \nabla_{\tau} \phi^{n} ], 1 \rangle 
		 \leq \frac{r_{n+1} \sqrt{r_{\max }}}{\left(1+r_{n+1}\right)} \frac{ \| \nabla_{\tau} \phi( t_{n} ) \|^{2} }{\tau_n} + 2 \big\langle G[ \nabla_{\tau} e_\phi^{n} ], 1 \big\rangle  =\mo ( \tau + h^4 ).
\end{equation*}
Moreover, it follows from the mean value theorem and \eqref{thmErrR:1} that
\begin{equation*} 
	\begin{aligned}
		\vert  E_{1h}[ \phi^{n} ] - R^n \vert & \leq \vert  E_{1h}[ \phi^{n} ] - E_{1h}[ \phi ( t_{n} ) ] \vert + \vert  E_{1h}[ \phi ( t_{n} ) ] - E_1[ \phi ( t_{n} ) ] \vert + \vert e_{R}^{n} \vert \\
		& \leq \vert  \langle F(\phi^n) - F( \phi( t_{n} ) ), 1 \rangle \vert + C ( \tau + h^2 ) \\
		& \leq C  ( \| e_{\phi}^{n} \|_{\infty} + \tau + h^2 ) =\mo ( \tau + h^2 ),
	\end{aligned}
\end{equation*}
where the $L^{\infty}$-norm error estimate \eqref{thmErrPhi:1} has been applied. Inserting the above estimates into \eqref{thmErrE:2} completes the proof of \eqref{thmErrE:0}.

Finally, from the discrete energy dissipation law \eqref{dis_energy:e1}, one immediately gets
\begin{equation*}
		\me_h[\phi^{n}, R^{n}] \leq \me_h[\phi^{n-1}, R^{n-1}] \quad \text{and} \quad \me_h[\phi^{n}, R^{n}] \leq \me_h[\phi^{0}, R^{0}] = E_h[\phi^{0}],
\end{equation*}
for all $ n \ge 1$, which together with \eqref{thmErrE:0} proves \eqref{thmErrE:1}.

\bibliographystyle{spmpsci}
\bibliography{Ref_JSC}

\end{document}